\documentclass[12pt,a4paper]{amsart}
\usepackage{amsfonts}
\usepackage{amsthm}
\usepackage{amsmath}
\usepackage{amscd}
\usepackage[latin2]{inputenc}
\usepackage{t1enc}
\usepackage[mathscr]{eucal}
\usepackage{indentfirst}
\usepackage{graphicx}
\usepackage{graphics}
\usepackage{pict2e}
\usepackage{epic}
\usepackage[colorlinks,linkcolor=red,anchorcolor=blue,citecolor=green,CJKbookmarks=True]{hyperref}
\usepackage{verbatim}
\usepackage[left=3.5cm, right=3.5cm, bottom=3cm]{geometry}
\usepackage{theoremref}
\usepackage{epstopdf}
\newtheorem{Th}{Theorem}[section]
\newtheorem{Lemma}[Th]{Lemma}
\newtheorem{Cor}[Th]{Corollary}

 \theoremstyle{definition}
\newtheorem{Def}[Th]{Definition}

\newtheorem{Rem}[Th]{Remark}
\newtheorem{?}[Th]{Problem}

\numberwithin{equation}{section}

\numberwithin{equation}{section}

\begin{document}

\bibliographystyle{amsalpha}

\title{ Rotational symmetry of uniformly 3-convex translating solitons of mean curvature flow in higher dimensions}

\begin{abstract}
  In this paper, we generalize the result of \cite{zhu2020so} to higher dimension. We prove that uniformly 3-convex translating solitons of  mean curvature flow in $\mathbb{R}^{n+1}$ which arise as blow up limit of embedded, mean convex mean curvature flow must have $SO(n-1)$ symmetry.
\end{abstract}

\author{Jingze Zhu}
\address{Department of Mathematics, Columbia University, New York, NY 10027}
\email{zhujz@math.columbia.edu}
\maketitle

\section{Introduction}
It is known that the blow up limit of an embedded, mean convex mean curvature
flow is a convex, noncollapsed ancient solution. This was first proved by the seminal work of White \cite{white2000size} \cite{white2003nature} and later streamlined by several authors (see \cite{sheng2009singularity} , \cite{andrews2012noncollapsing}, \cite{haslhofer2017mean}).

Recently, lots of work have been done in studying the convex noncollapsed ancient solutions. For example, in the uniform two convex
case, the classification is complete: Haslhofer \cite{haslhofer2015uniqueness} proved that the only convex, uniformly two convex, noncollapsed translator is the Bowl soliton, Brendle and Choi \cite{brendle2018uniqueness} \cite{brendle2019uniqueness} proved that the only strictly convex, uniformly two convex, noncollapsed noncompact ancient solution must be the Bowl soliton.
Angenent, Daskalopoulos and Sesum \cite{angenent2020uniqueness} proved that the only convex, uniformly two convex, noncollapsed compact solution which is not the shrinking sphere must be  the unique ancient Oval. The uniqueness is up to scaling
and translation.

In the uniformly 3-convex case, however, little classification results are known so far.
In the previous paper \cite{zhu2020so}, we proved that the convex, uniformly 3-convex, noncollapsed translator has $SO(2)$ symmetry in $\mathbb{R}^4$.

In this paper, we generalize this result to higher dimensions, here is the main theorem:
\begin{Th}\label{Main Theorem}
Suppose that $M^n\subset \mathbb{R}^{n+1}$
is a complete, noncollapsed, convex, uniformly 3-convex smooth
translating soliton of mean curvature flow with a tip that attains maximal mean
curvature. Then M has SO(n-1) symmetry.
\end{Th}

The proof has many similarities with \cite{zhu2020so} and \cite{brendle2018uniqueness}. For readers' convenience and self consistency,
we follow these two papers closely and present necessary details as much as possible.

This paper is organized as follows. In section 3,  we discuss about the symmetry improvement, which is inspired by the Neck Improvement Theorem of \cite{brendle2018uniqueness} \cite{brendle2019uniqueness}. The spirit is to  show
that if a solution of mean curvature flow is $\epsilon$ symmetric in a large parabolic neighborhood, then it is $\epsilon/2$ symmetric at the center point.
We  prove Cylindrical Improvement Theorem
\ref{Cylindrical improvement}, this is a direct generalization of the Neck Improvement Theorem.
We need to analyze the 2D equation instead of 1D heat equation.
In the uniform 3-convex case we need an additional symmetry improvement (Theorem \ref{bowl x R improvement}) that works for Bowl$\times\mathbb{R}$, just as in our previous paper \cite{zhu2020so}.
The strategy is inspired from \cite{brendle2018uniqueness} and
\cite{angenent2020uniqueness}, we can iterate the Cylindrical Improvement to get much better symmetry along the boundary and use the barrier to control the symmetry
near the center of the parabolic neighborhood. However in our case the
boundary is not as simple. In fact there is some portions of the boundary on which
the symmetry don't improve. It is overcome by carefully choosing barrier functions.

In section 4, we argue exactly as in \cite{zhu2020so},
using the ingredients established in section 3.
First we prove the canonical neighborhood Lemma (Lemma \ref{canonical nbhd lemma}, \ref{blowdownnecklemma}). They tells us that away from a compact set, every point
lies on a large parabolic neighborhood that resembles $S^{n-2}\times\mathbb{R}^2$
or Bowl$\times\mathbb{R}$ if the soliton is not uniformly 2-convex. We briefly explain how this is done: first assume that the conclusion is not true,
then we find the contradicting sequence and by rescaling and passing to subsequential limit we obtain an convex noncollapsed, uniformly 3-convex ancient
solution that contains a line, and thus splits. The results of \cite{white2000size} \cite{white2003nature}, \cite{sheng2009singularity}, \cite{haslhofer2017mean} and
the  maximal principal \cite{hamilton1986four} are used crucially.
If we modulo the splitting, the remain solution is uniformly two convex,
thus is either compact or is
the Bowl soliton or $S^{n-2}\times\mathbb{R}$, by the classification \cite{brendle2018uniqueness}. The latter two noncompact models leads to the contradiction whereas we can exploit the translator equation (which in particular implies that the solution is a graph over a hyperplane) to rule out the compact models.
The final step is to combine the above ingredients and follow the argument of
Theorem 5.2 in \cite{brendle2018uniqueness} to finish the proof of the main theorem.

  \textbf{Acknowledgement:} The author would like to thank his advisor Simon Brendle for his helpful discussions and encouragement.\newline

\section{Preliminary}
In this section we give some definitions and basic facts about the mean curvature flow.

Recall that the mean curvature flow is a family of embeddings: $F_t: M^n\rightarrow\mathbb{R}^{n+1}$  which satisfies
\begin{align*}
  \frac{\partial F_t}{\partial t} & =\vec{H}
\end{align*}
Denote $M_t=F_t(M)$. A mean curvature solution is called ancient if $F_t$ exists on $t\in(-\infty,T)$ for some $T$.

If $F_t$ are complete and oriented, we can fix an orientation and globally define the normal vector $\nu$.
Then the mean curvature $H$ is defined to be $\vec{H}=H\nu$. 
The mean curvature flow is called \textbf{(strictly) mean convex}, if $H\geq0$ $(H>0) $ along the flow.

A mean convex mean curvature flow solution $M_t^n\subset\mathbb{R}^{n+1}$
  is called \textbf{uniformly $k$-convex}, if there is a positive constant $\beta$
  such that
  \begin{align*}
    \lambda_1+...+\lambda_k\geq \beta H
  \end{align*}
  along the flow, where $\lambda_1\leq\lambda_2\leq ...\leq \lambda_n$ are principal curvatures.
  
  In particular, any mean convex surface is uniformly $n$-convex, and being uniformly  $1$-convex is equivalent to being uniformly convex. 

An important special case for ancient  solution is the \textbf{translating solution}, which is  characterized by the equation
\begin{align*}
   H = \left<V,\nu\right>
\end{align*}
for some fixed nonzero vector $V$.  In this paper 
we usually use $\omega_{n}$ in place of $V$, where $\omega_n$ is a unit vector in the direction of the $x_n$ axis.

The family of surfaces $M_t=M_0+tV$ is a mean curvature flow provided that $M_0$ satisfies the translator equation.

For a point $x$ on a hypersurface $M^n\subset\mathbb{R}^{n+1}$ and radius $r$, we use $B_r(x)$ to denote the Euclidean ball and use $B_{g}(x,r)$ to denote the geodesic ball with respect to the metric $g$ on $M$ induced by the embedding.
for a space-time point $(\bar{x},\bar{t})$ in a mean curvature flow solution, $\hat{\mathcal{P}}(\bar{x},\bar{t},L,T)=B_{g(\bar{t})}(\bar{x},LH^{-1})\times [\bar{t}-TH^{-2},\bar{t}]$ where $H=H(\bar{x},\bar{t})$ (c.f \cite{huisken2009mean} pp.188-190) ).

\section{Symmetry Improvement}
\begin{Def}
A collection of vector fields $\mathcal{K} = \{K_{\alpha} : 1 \leq \alpha \leq \frac{(n-1)(n-2)}{2} \}$ in $\mathbb{R}^{n+1}$ is called normalized set of rotation vector fields, if there exists a matrix $S \in O(n+1)$, a set of orthonormal basis
$\{J_{\alpha} : 1 \leq \alpha \leq \frac{(n-1)(n-2)}{2} \}$ of $so(n-1) \subset so(n+1)$
and a point $q \in \mathbb{R}^{n+1}$ such that
\begin{align*}
  K_{\alpha}(x) & = SJ_{\alpha}S^{-1}(x-q)
\end{align*}
where we use the inner product $\left<A,B\right> = \text{Tr}(AB^{T})$ for matrices $A, B \in so(n-1)$.
\end{Def}

\begin{Rem}
  The choice of $S,q$ is not unique. In other words, we can always find different
  $S,q$ that represent the same normalized set of rotation vector field.
\end{Rem}

\begin{Def}
  Let $M_t$ be a mean curvature flow solution. We say that the space time point $(\bar{x}, \bar{t})$ is $(\epsilon, R)$ cylindrical if the parabolic neighborhood $\hat{P}(\bar{x},\bar{t}, R^2, R^2)$ is $\epsilon$ close
  (in $C^{10}$ norm) to a family of shrinking cylinders
  $S^{n-2}\times\mathbb{R}^2$ after a parabolic rescaling such that $H(\bar{x}, \bar{t}) = \sqrt{\frac{n-2}{2}}$.
\end{Def}

\begin{Def}
  Let $M_t$ be a mean curvature flow solution. We say that a space time point
  $(\bar{x}, \bar{t})$ is $\epsilon$ symmetric if there exists a normalized set of rotation vector fields
  $\mathcal{K} = \{K_{\alpha} : 1 \leq \alpha \leq \frac{(n-1)(n-2)}{2}\}$ such that $\max_{\alpha}|\left<K_{\alpha}, \nu\right> |H \leq \epsilon$ and
  $\max_{\alpha}|K_{\alpha}|H \leq 5n$ in the parabolic neighborhood
  $\hat{P}(\bar{x}, \bar{t}, 100n^{5/2}, 100^2n^5)$
\end{Def}

\begin{Lemma}\label{Lin Alg Cylinder A}
  Let $\mathcal{K}=\{K_{\alpha},1\leq\alpha\leq \frac{(n-1)(n-2)}{2}\}$
  be a normalized set of rotation vector fields in the form of
  \begin{align*}
    K_{\alpha}=SJ_{\alpha}S^{-1}(x-q)
  \end{align*}
  where
  $\{J_{\alpha}, 1\leq\alpha\leq \frac{(n-1)(n-2)}{2}\}$ is an
  orthonormal basis of $so(n-1)\subset so(n+1)$ and $S\in O(n+1), q\in\mathbb{R}^{n+1}$.
  Suppose that either one of the following happens:
  \begin{enumerate}
    \item on the cylinder $S^{n-2}\times\mathbb{R}^2$ in $\mathbb{R}^{n+1}$ (for which the
  $S^{n-2}$ factor has radius 1) :
  \begin{itemize}
     \item $\left<K_{\alpha},\nu\right> = 0$ in $B_{g}(p, 1)$ for each $\alpha$
     \item $\max_{\alpha}|K_{\alpha}| H \leq 5n$ in
      $B_{g}(p, 10n^{2})$
   \end{itemize}
    \item on $Bowl^{n-1}\times\mathbb{R}$ in $\mathbb{R}^{n+1}$:
  \begin{itemize}
     \item $\left<K_{\alpha},\nu\right> = 0$ in $B_{g}(p, 1)$ for each $\alpha$
     \item $\max_{\alpha}|K_{\alpha}| H \leq 5n$ at $p$
   \end{itemize}
  \end{enumerate}
  where $\nu$ is the normal vector, $g$ is the induced metric,
  $p$ is arbitrary point on either one of the model.
  Then $S,q$ can be chosen in such a way:
  \begin{itemize}
    \item $S\in O(n-1)\subset O(n+1)$
    \item $q=0$
  \end{itemize}
  In particular, for any orthonormal basis $\{J'_{\alpha}, 1\leq\alpha\leq \frac{(n-1)(n-2)}{2}\}$ of $so(n-1)$ there is a basis transform matrix
  $\omega\in O(\frac{(n-1)(n-2)}{2})$ such that
  $K_{\alpha}  = \sum_{\beta=1}^{\frac{(n-1)(n-2)}{2}}\omega_{\alpha\beta} J'_{\beta}x$
\end{Lemma}

\begin{proof}

\textbf{Case (1)}  Let's first assume that $n\geq 4$, we use the coordinate $(x_1,...x_{n+1})$ in $\mathbb{R}^{n+1}$ and let the cylinder $S^{n-2}\times\mathbb{R}^2$ be
 represented by $\{x_1^2+...+x_{n-1}^2=1\}$

 Now $\langle K_{\alpha},\nu\rangle =0$ is equivalent to:
 \begin{align}\label{Lin Alg 1}
   \nu^{T}SJ_{\alpha}S^{-1}(x-q)=0
 \end{align}
 without loss of generality, we may choose $q$ such that
 \begin{align}\label{LinAlg Condition for q}
   q\perp \bigcap_{\alpha}\ker J_{\alpha}S^{-1}
 \end{align}

 Note that $SJ_{\alpha}S^{-1}$ is antisymmetric and $\nu^{T}=(x_1,...,x_{n-1},0,0)$
 on $S^{n-2}\times\mathbb{R}^2$, (\ref{Lin Alg 1}) is equivalent to
 \begin{align}\label{LinAlg2}
   \sum_{1\leq i\leq n-1} x_i \left(\sum_{l=n,n+1}(SJ_{\alpha}S^{-1})_{il}x_l - (SJ_{\alpha}S^{-1}q)_{i}\right) =  0
 \end{align}
 Since this holds on an open set of the cylinder, we obtain that
 \begin{align}
   (SJ_{\alpha}S^{-1})_{il} =&0, \text{ \  } i=1,...,n-1,\ l=n,n+1 \label{LinAlg3}\\
   (SJ_{\alpha}S^{-1}q)_{i} =&0, \text{ \  } i=1,...,n-1\label{LinAlg4}
 \end{align}
 holds for each $1\leq\alpha \leq\frac{(n-1)(n-2)}{2} $.

 Using the fact that $\{J_{\alpha}, 1\leq\alpha\leq \frac{(n-1)(n-2)}{2} \}$ is an orthonormal basis of $so(n-1)$
 in (\ref{LinAlg3}), we get:
 \begin{align*}
   S_{ij}S^{-1}_{kl} - S_{ik}S^{-1}_{jl} =0
 \end{align*}
 or equivalently (since $S^{-1}=S^{T}$ or $S^{-1}=-S^{T}$)
  \begin{align}\label{LinAlg5}
   S_{ij}S_{lk} - S_{ik}S_{lj} =0
 \end{align}
 for all $1\leq i,j,k\leq n-1$ and $l=n,n+1$.

 Next, since $n-1\geq 3$, we can choose nontrivial $(y_1,...y_{n-1})$ such that
 for $l=n,n+1$
 \begin{align}\label{LinAlg6}
   \sum_{k=1}^{n-1}S_{lk}y_k =0
 \end{align}
 Now we multiply (\ref{LinAlg5}) by $y_k$ and sum over $k=1,...,n-1$ to obtain that, for $1\leq i,j\leq n-1$ and $l=n,n+1$:
 \begin{align}\label{Lin Alg 7}
   \sum_{k=1}^{n-1}S_{ik}y_k S_{lj} = 0
 \end{align}
 The invertibility of $S$ implies that
 \begin{align}\label{Lin Alg 9}
   \sum_{k=1}^{n-1}S_{ik}y_k \neq 0
 \end{align}
 holds for at least one $i\leq n+1$. (\ref{LinAlg6}) implies that (\ref{Lin Alg 9}) can not hold for $i=n,n+1$, thus must hold for some $i\leq n-1$.
 At this point, (\ref{Lin Alg 7}) implies that
 \begin{align}\label{Lin Alg 8}
   S_{lj}=0
 \end{align}
 for all $l=n,n+1$ and $j=1,...,n-1$. That means $S$ preserves the direct sum decomposition, therefore we may choose $S\in O(n-1)\subset O(n+1)$.
 Moreover, (\ref{LinAlg4}) and (\ref{LinAlg Condition for q}) implies that $q=0$.
 Finally,$\{SJ_{\alpha}S^{-1}, 1\leq\alpha\leq \frac{(n-1)(n-2)}{2} \}$ is still an orthonormal basis of $so(n-1)$, therefore we can find basis transform matrix
 $\omega\in O(\frac{(n-1)(n-2)}{2})$ such that
  $K_{\alpha}  = \sum_{\beta=1}^{\frac{(n-1)(n-2)}{2}}\omega_{\alpha\beta} J'_{\beta}x$ for each given orthonormal basis $\{J'_{\alpha}, 1\leq\alpha\leq\frac{(n-1)(n-2)}{2}\}$ of $so(n-1)$.

  If $n=3$, the argument in Lemma 3.5 of \cite{zhu2020so} applies.
  However we give a brief proof here:
  Note that (\ref{LinAlg3}), (\ref{LinAlg4}) still applies. Since $so(2)$ is one dimensional, $\alpha\equiv 1$ and $J_{\alpha}=J$ is a rank 2 matrix.
  Since $SJS^{-1}$ is antisymmetric and rank 2, (\ref{LinAlg3}) implies that $SJS^{-1}=J$ or $J'$,
   where $J'$ is a rank 2 matrix with only two nonzero entries $J'_{34}=-J'_{43}=\frac{1}{\sqrt{2}}$.
   If $SJS^{-1}=J'$, then $|K|H\geq 5n^2>5n$, a contradiction. So $SJS^{-1}=J$ and (\ref{LinAlg Condition for q})(\ref{LinAlg3}) implies that $q=0$. Clearly $S$ can be chosen to be Id.

  \textbf{Case (2)} Note that on Bowl$\times\mathbb{R}$, we have $\nu^{T}=(\frac{u'(\bar{x})}{|\bar{x}|}\bar{x},-1,0)/\sqrt{1+u'^2}$ and $x_n=u(|\bar{x}|)$
 where $\bar{x}=(x_1,...,x_{n-1})$ and $u$ is the solution to the ODE:
  \begin{align*}
    \frac{u''}{1+u'^2}+\frac{(n-1)u'}{x}=1
  \end{align*}
  with initial condition $u(0)= u'(0)=0$. Note that $\frac{u'(r)}{r}$ is smooth even at $r=0$.

  (\ref{Lin Alg 1}) is equivalent to:
 \begin{align}\label{LinAlg II.2}
   \sum_{1\leq i\leq n-1} x_i &\left(-(SJ_{\alpha}S^{-1})_{ni}+\sum_{l=n,n+1}\frac{u'(\bar{x})}{|\bar{x}|}(SJ_{\alpha}S^{-1})_{il}x_l - \frac{u'(\bar{x})}{|\bar{x}|}(SJ_{\alpha}S^{-1}q)_{i}\right) \\
   -& (SJ_{\alpha}S^{-1})_{n,n+1}x_{n+1}+(SJ_{\alpha}S^{-1}q)_{n}=  0
 \end{align}
 Since this holds on an open set of a Bowl$\times\mathbb{R}$, we can first fix $|\bar{x}|, x_n, x_{n+1}$ and $(x_1,...,x_{n-1})$ moves freely on a small open set of $S^{n-2}_{|\bar{x}|}$, therefore:
 \begin{align}\label{Lin Alg II.1}
   -(SJ_{\alpha}S^{-1})_{ni}+\sum_{l=n,n+1}\frac{u'(\bar{x})}{|\bar{x}|}(SJ_{\alpha}S^{-1})_{il}x_l - \frac{u'(\bar{x})}{|\bar{x}|}(SJ_{\alpha}S^{-1}q)_{i}=0
 \end{align}
 for each $1\leq i\leq n-1$. Moreover,
  \begin{align}\label{Lin Alg II.4}
   -(SJ_{\alpha}S^{-1})_{n,n+1}x_{n+1}+(SJ_{\alpha}S^{-1}q)_{n}=0
 \end{align}

 Then we can let $x_{n+1}$ move freely in a small open set for (\ref{Lin Alg II.1}) (\ref{Lin Alg II.4}), since $u'(r)/r$ is never 0, we  get

 \begin{align}\label{Lin Alg II.5}
   (SJ_{\alpha}S^{-1})_{i,n+1}=&0
 \end{align}
 for  all $ i\leq n$.

 Next, $\frac{u'(r)}{r}$ is not a constant function in any open interval, therefore
 \begin{align}\label{Lin Alg II.6}
      (SJ_{\alpha}S^{-1})_{in}=&0 \\
      (SJ_{\alpha}S^{-1}q)_i=&0
 \end{align}
 for all $i\leq n-1$.

 Consequently, (\ref{LinAlg3}) and  (\ref{LinAlg4}) holds, and the conclusion follows immediately by the same argument as in case (1).
\end{proof}

Using similar computations and some basic linear algebra, we also have the following:
\begin{Lemma}\label{Lin Alg Cylinder B}
  Let  $\{J_{\alpha}, 1\leq\alpha\leq \frac{(n-1)(n-2)}{2}\}$ be an
  orthonormal basis of $so(n-1)\subset so(n+1)$ and $A\in so(n+1)$ such that
  $A\perp so(n-1)\oplus so(2)$.
  Suppose that either of the following scenarios happens:
  \begin{enumerate}
    \item on the cylinder $S^{n-2}\times\mathbb{R}^2$ in $\mathbb{R}^{n+1}$ (for which the
  $S^{n-2}$ factor has radius 1) we have:
    \begin{itemize}
      \item $\left<[A,J_{\alpha}]x+c_{\alpha},\nu\right> = 0$ in $B_{g}(p, 1)$ for each $\alpha$
    \end{itemize}

  \item on $Bowl^{n-1}\times\mathbb{R}$ in $\mathbb{R}^{n+1}$  we have:
    \begin{itemize}
     \item $\left<[A, J_{\alpha}]x+c_{\alpha},\nu\right> = 0$ in $B_{g}(p, 1)$ for each $\alpha$
    \end{itemize}
  \end{enumerate}
  where $\nu$ is the normal vector, $g$ is the induced metric,
  $p$ is arbitrary point on either one of the model.
  Then $A=0$ and $c_{\alpha}=0$.

\end{Lemma}

\begin{Lemma}\label{Vector Field closeness on Cylinder}
  There exists constants $0 < \epsilon_{c} \ll 1$ and $C>1$ depending only on $n$
  with the following properties. Let $M$ be a hypersurface in $\mathbb{R}^{n+1}$
  which is $\epsilon_c$ close (in $C^3$ norm) to
  a geodesic ball in
  $ S_{\sqrt{2(n-2)}}^{n-2}\times \mathbb{R}^2 $
   of radius $20n^{5/2}\sqrt{\frac{2}{n-2}}$ and
  $\bar{x}\in M $ be a point that is $\epsilon_c$ close to $S_{\sqrt{2(n-2)}}^{n-2}\times \{0\}$.
  Suppose that $\epsilon \leq \epsilon_c$ and
  $\mathcal{K}^{(1)} = \{K^{(1)}_{\alpha}: 1 \leq \alpha \leq \frac{(n-1)(n-2)}{2}\}$,
  $\mathcal{K}^{(2)} = \{K^{(2)}_{\alpha}: 1 \leq \alpha \leq \frac{(n-1)(n-2)}{2}\}$, are two normalized set of rotation vector fields, assume that:
  \begin{itemize}
    \item $\max_{\alpha}|\left<K_{\alpha}^{(i)}, \nu\right>|H \leq \epsilon$ in
    $B_g(\bar{x}, H^{-1}(\bar{x}))\subset M$
    \item $\max_{\alpha}|K_{\alpha}^{(i)}|H \leq 5n$ in $B_{g}(\bar{x}, 10n^{5/2}H^{-1}(\bar{x}))\subset M$
  \end{itemize}
  for $i = 1 , 2$, where $g$ denotes the induced metric on $M$ by embedding.
  Then for any $L > 1$:
  \begin{align*}
    \inf\limits_{\omega\in O(\frac{(n-1)(n-2)}{2})}\sup\limits_{B_{LH(\bar{x})^{-1}}(\bar{x})}\max_{\alpha}
    |K_{\alpha}^{(1)}-\sum_{\beta = 1}^{\frac{(n-1)(n-2)}{2}}\omega_{\alpha\beta}K^{(2)}_{\beta}|H(\bar{x})
    \leq CL\epsilon
  \end{align*}
\end{Lemma}
\begin{proof}
  The proof is analogous to \cite{brendle2019uniqueness},\cite{zhu2020so}. We only consider the case that $L=10$, for general $L$ use the fact that
  $K^{(i)}$ are affine functions. Throughout the proof, the constant $C$ depends only on $n$.

  Argue by contradiction, if the conclusion is not true.
  Then there exists a sequence of pointed hypersurfaces $(M_j, p_j)$
  that are $\frac{1}{j}$ close to a geodesic ball of radius $20n^{5/2}H_0^{-1}$ in $S_{\sqrt{2(n-2)}}^{n-2}\times \mathbb{R}^2 $ and $|p_j - (\sqrt{2(n-2)},0,...,0)| \leq \frac{1}{j}$,
   where $H_0 = \sqrt{\frac{n - 2}{2}}$ is the mean curvature of $S_{\sqrt{2(n-2)}}^{n-2}\times \mathbb{R}^2$
   and $g_j$ is the metric on $M_j$ induced by embedding.
   Moreover, there exists normalized set of rotation vector fields
   $\mathcal{K}^{(i,j)} = \{K^{(i,j)}_{\alpha}: 1\leq\alpha\leq \frac{(n-1)(n-2)}{2}\}$  and $\epsilon_j\leq\frac{1}{j}$ such that
   \begin{itemize}
     \item $\max_{\alpha}|\left<K^{(i,j)}_{\alpha}, \nu\right>|H\leq\epsilon_j$
      in $B_{g_j}(p_j, H(p_j)^{-1})$
     \item $\max_{\alpha}|K_{\alpha}^{(i,j)}|H\leq 5n$ in
      $B_{g_j}(p_j, 10n^{5/2}H(p_j)^{-1})$
   \end{itemize}
   for $i = 1,2$, but
   \begin{itemize}
     \item $\inf\limits_{\omega\in O(\frac{(n-1)(n-2)}{2})}\sup\limits_{B_{LH(\bar{x})^{-1}}(\bar{x})}\max_{\alpha}
    |K_{\alpha}^{(1,j)}-\sum_{\beta = 1}^{\frac{(n-1)(n-2)}{2}}\omega_{\alpha\beta}K^{(2,j)}_{\beta}|H(\bar{x})
    \geq j\epsilon_j$
   \end{itemize}

   Therefore, $M_j$ converges to $M_{\infty} = B_{g_{\infty}}(p_{\infty}, 20n^{5/2}H_0^{-1})\subset S_{\sqrt{2(n-2)}}^{n-2}\times \mathbb{R}^2 $
   and $H(p_j)\rightarrow \sqrt{\frac{n-2}{2}}$, where $p_{\infty}=(\sqrt{2(n-2)},0,...,0)$
   and $g_{\infty}$ denoted the induced metric on
   $S_{\sqrt{2(n-2)}}^{n-2}\times \mathbb{R}^2 $.

   Suppose that for each $i=1,2$ and $j\geq 1$,
   $K^{(i,j)}_{\alpha}(x) = S_{(i,j)}J^{(i,j)}_{\alpha}S_{(i,j)}^{-1}(x-b_{(i,j)})$
   where $S_{(i,j)}\in O(n+1)$ and $b_{(i,j)}\in\mathbb{R}^{n+1}$,
   $\{J^{(i,j)}_{\alpha}: 1\leq\alpha \leq\frac{(n-1)(n-2)}{2}\}$ is
    orthonormal basis of
   $so(n-1)\subset so(n+1)$. Without loss of generality we may assume that
   $\displaystyle b_{(i,j)}\perp \bigcap_{\alpha}\ker(J^{(i,j)}_{\alpha}S_{(i,j)}^{-1})$.
   Then
   \begin{align*}
     |b_{(i,j)}| \leq C\sum_{\alpha}|S_{(i,j)}J^{(i,j)}_{\alpha}S_{(i,j)}^{-1} b_{(i,j)}| \leq C(n)
   \end{align*}
   Therefore we can pass to a subsequence such that $S_{(i,j)}\rightarrow S_{(i,\infty)}$, $J_{\alpha}^{(i,j)}\rightarrow J_{\alpha}^{(i,\infty)}$
   and $b_{(i,j)}\rightarrow b_{(i,\infty)}$ for each $i,\alpha$.
   Consequently $K^{(i,j)}\rightarrow K^{(i,\infty)}_{\alpha} = S_{(i,\infty)}J_{\alpha}^{(i,\infty)}S_{(i,\infty)}^{-1}(x-b_{(i,\infty)})$ for $i=1,2$.

   The convergence implies that
   \begin{itemize}
     \item $\left<K^{(i,\infty)}_{\alpha},\nu\right> = 0$ in $B_{g_{\infty}}(p_{\infty,}, H_0^{-1})$ for each $\alpha$
     \item $\max_{\alpha}|K^{(i,\infty)}_{\alpha}| H \leq 5n$ in
      $B_{g_{\infty}}(p_{\infty,}, 10n^{5/2}H_0^{-1})$
   \end{itemize}
   Let's fix an orthonormal basis $\{J_{\alpha}: 1\leq\alpha \leq\frac{(n-1)(n-2)}{2}\}$  of $so(n-1)$. By Lemma \ref{Lin Alg Cylinder A} we have for each $\alpha$:
    \begin{flalign*}
      K_{\alpha}^{(i,\infty)}(x) = \sum_{\beta=1}^{\frac{(n-1)(n-2)}{2}}\omega^{(i,\infty)}_{\alpha\beta} J_{\beta}x
    \end{flalign*}
    for some $\omega^{(i,\infty)}\in O(\frac{(n-1)(n-2)}{2})$ .
    In particular we have
    \begin{align}\label{sjs-1=omega J}
      S_{(i,\infty)}J^{(i,\infty)}_{\alpha}S_{(i,\infty)}^{-1} = \sum_{\beta=1}^{\frac{(n-1)(n-2)}{2}}\omega^{(i,\infty)}_{\alpha\beta} J_{\beta}
    \end{align}
    and $b_{(i,\infty)}=0$ for $i=1,2$ (Note that $b_{(i,\infty)}\perp\bigcap_{\alpha}\ker(J^{(i,\infty)}_{\alpha}S^{-1}_{(i,\infty)}$ ).

    Since $\{J^{(i,\infty)}_{\alpha}, 1\leq\alpha\leq \frac{(n-1)(n-2)}{2}\}$ is an orthonormal basis of $so(n-1)$, we can find $\eta^{(i,j)}\in O(\frac{(n-1)(n-2)}{2})$ such that
    $J_{\alpha}^{(i,j)}=\sum_{\beta}\eta^{(i,j)}_{\alpha\beta}J_{\beta}^{(i,\infty)}$ . Then for each $i,j$
    \begin{align}\label{sjs-1=omega J 2}
      S_{(i,\infty)}J^{(i,j)}_{\alpha}S_{(i,\infty)}^{-1}=&\sum_{\beta}\eta^{(i,j)}_{\alpha\beta}S_{(i,\infty)}J_{\beta}^{(i,\infty)}S_{(i,\infty)}^{-1} \notag \\
      =& \sum_{\beta,\gamma}\eta^{(i,j)}_{\alpha\beta}\omega^{(i.\infty)}_{\beta\gamma}J_{\gamma}
    \end{align}
     Now (\ref{sjs-1=omega J 2}) means that $\{S_{(i,\infty)}J^{(i,j)}_{\alpha}S_{(i,\infty)}^{-1}, 1\leq\alpha\leq \frac{(n-1)(n-2)}{2}\}$ is an orthonormal basis of $so(n-1)$.

    Without loss of generality, we may assume that
    $S_{(1,\infty)}=S_{(2,\infty)} =Id$ and $\omega^{(1,\infty)}_{\alpha\beta} = \delta_{\alpha\beta}$, for otherwise we may replace $S_{(1,j)}$
    by $S_{(i,j)}S_{(i,\infty)}^{-1}$, replace $J_{\alpha}^{(i,j)}$ by $S_{(i,\infty)}J_{\alpha}^{(i,j)}S_{(i,\infty)}^{-1}$ for $i=1,2$
    and replace $J_{\alpha}$ by
    $\sum_{\beta=1}^{\frac{(n-1)(n-2)}{2}}\omega^{(1,\infty)}_{\alpha\beta} J_{\beta}$.

    Therefore, $S_{(1,j)}^{-1}S_{(2,j)}$ is close to Id for large $j$. Then we can find $S_j\in O(n+1)$
    and $A_j\in so(n+1)$ such that
    \begin{itemize}
     \item $S_{(1,j)}^{-1}S_{(2,j)} = \exp(A_j)S_j $
      \item $S_j$ preserves the direct sum decomposition $\mathbb{R}^{n-1}\oplus\mathbb{R}^2$
      \item $A_j\perp so(n-1)\oplus so(2)$
      \item $S_j\rightarrow Id$ and $A_j\rightarrow 0$
    \end{itemize}

   Since $S_j$ preserves the direct sum decomposition,
   we can  find basis transform matrix ${\omega}^{(j)}\in O(\frac{(n-1)(n-2)}{2})$ such that
   \begin{align*}
     S_j^{-1} J^{(1,j)}_{\alpha}S_{j} = \sum_{\beta=1}^{\frac{(n-1)(n-2)}{2}}{\omega}^{(j)}_{\alpha\beta} J_{\beta}^{(2,j)}
   \end{align*}
   for every $j$ and $\alpha$. Equivalently,
   \begin{align*}
      J^{(1,j)}_{\alpha} = \sum_{\beta=1}^{\frac{(n-1)(n-2)}{2}}{\omega}^{(j)}_{\alpha\beta} S_{j}J_{\beta}^{(2,j)}S_j^{-1}
   \end{align*}

   To see what the definition of $\omega^{(j)}$ implies, we compute the following:
   \begin{align*}
     \sum_{\beta=1}^{\frac{(n-1)(n-2)}{2}} \omega^{(j)}_{\alpha\beta}K_{\beta}^{(2,j)}(x) = & \sum_{\beta} \omega^{(j)}_{\alpha\beta}S_{(2,j)}J^{(2,j)}_{\beta}S_{(2,j)}^{-1}(x-b_{(2,j)})\\
     =& \sum_{\beta}\omega^{(j)}_{\alpha\beta} S_{(1,j)}\exp(A_j) S_j J_{\beta}^{(2,j)}S_j^{-1}\exp(-A_j)S_{(1,j)}^{-1} (x-b_{(2,j)}) \\
     =&  S_{(1,j)}\exp(A_j)J^{(1,j)}_{\alpha}\exp(-A_j)S_{(1,j)}^{-1}(x-b_{(2,j)})
   \end{align*}

     For each $\alpha$, define:
     \begin{align*}
       W_{\alpha}^j  = \frac{K_{\alpha}^{(1,j)} -\sum_{\beta} \omega^{(j)}_{\alpha\beta}K_{\beta}^{(2,j)} }{\sup\limits_{B_{10H(p_j)^{-1}}(p_j)}\max_{\alpha}|K_{\alpha}^{(1,j)} -\sum_{\beta} \omega^{(j)}_{\alpha\beta}K_{\beta}^{(2,j)} |}
     \end{align*}

     Let
     \begin{itemize}
       \item $P_{\alpha}^j = S_{(1,j)}[J^{(1,j)}_{\alpha}-\exp(A_j)J^{(1,j)}_{\alpha}\exp(-A_j)]S_{(1,j)}^{-1}$
       \item $c_{\alpha}^j = -S_{(1,j)}J^{(1,j)}_{\alpha}S_{(1,j)}^{-1}(b_{(1,j)}-b_{(2,j)})$
       \item $Q_j$ = $\sup\limits_{B_{10H(p_j)^{-1}}(p_j)}\max_{\alpha}|K_{\alpha}^{(1,j)} -\sum_{\beta} \omega^{(j)}_{\alpha\beta}K_{\beta}^{(2,j)} |$
     \end{itemize}

     By the above discussion we can find that for each $\alpha$:
     \begin{align*}
        K_{\alpha}^{(1,j)} -\sum_{\beta=1}^{\frac{(n-1)(n-2)}{2}} \omega^{(j)}_{\alpha\beta}K_{\beta}^{(2,j)}  &= P^j_{\alpha}(x-b_{(2,j)}) + c^j_{\alpha}
     \end{align*}

     By definition of $Q_j$ we have $|P^j_{\alpha}| + |c^j_{\alpha}| \leq CQ_j$. Consequently, for sufficiently large $j$:

     \begin{align*}
       & |P_{\alpha}^j| = |[A_j, J^{(1,j)}_{\alpha}] + o(|A_j|)| \leq CQ_j \\
      \Rightarrow & |A_j|\leq C\max_{\alpha}|[A_j,J^{(1,j)}_{\alpha}]| \leq CQ_j+o(|A_j|)\\
      \Rightarrow & |A_j|\leq CQ_j
     \end{align*}
     The second inequality used the fact that $A_j \perp so(n-1)\oplus so(2)$.

     Note that $J^{(1,j)}_{\alpha}\rightarrow J_{\alpha}$ by the previous discussion. Now we can pass to a subsequence such that $\frac{P_{\alpha}^j}{Q_j}\rightarrow [A,J_{\alpha}]$ and
     $\frac{c^j_{\alpha}}{Q_j}\rightarrow c_{\alpha}\in Im(J_{\alpha})$, .  Consequently:
     \begin{align*}
       W^j_{\alpha} \rightarrow W_{\alpha}^{\infty} = [A,J_{\alpha}]x +c_{\alpha}
     \end{align*}
     note that $W^{\infty}_{\alpha}$ are not all $0$, because
     $\sup\limits_{B_{10H(p_j)^{-1}}(p_j)}\max_{\alpha}|W^{\infty}_{\alpha}|=1$.

     On the other hand, by assumption
     \begin{itemize}
       \item $\max_{\alpha}|\left<K^{(1,j)}_{\alpha} - \sum_{\beta}\omega^{(j)}_{\alpha\beta}K_{\beta}^{(2,j)}, \nu\right>|\leq 2H^{-1}\epsilon_j$
       \item ${\sup\limits_{B_{10H(p_j)^{-1}}(p_j)}\max_{\alpha}|K_{\alpha}^{(1,j)} -\sum_{\beta} \omega^{(j)}_{\alpha\beta}K_{\beta}^{(2,j)} |\geq jH^{-1}}\epsilon_j$
     \end{itemize}
     therefore in the limit $\max_{\alpha}|\left<\nu, W_{\alpha}^{\infty}\right>|=0$ in
     $B_{g_{\infty}}(p_{\infty},1)$.

     By Lemma \ref{Lin Alg Cylinder B}, $W_{\alpha}^{\infty}\equiv 0$ for all $\alpha$, a contradiction.

\end{proof}

\begin{Lemma}\label{VF close Bowl times R}
  Given $\delta < 1$, there exists $0<\epsilon_b \ll 1$ and $C>1$ depending
  only on $n$ and $\delta$ with the following properties. Let $\Sigma^{n-1}\subset\mathbb{R}^n$ be the Bowl solition with maximal
  mean curvature 1 and $M\subset\mathbb{R}^{n+1}$ be a hypersurface with induced
  metric $g$. Suppose that  $q\in \Sigma\times\mathbb{R}$,
  and $M$ is a graph over the geodesic ball in $\Sigma\times\mathbb{R}$ of radius
  $2H(q)^{-1}$ centered at $q$. After rescaling by $H(q)^{-1}$ the graph $C^3$ norm is no more than $\epsilon_b$.
  Let $\bar{x}\in M$ be a point that has rescaled distance to $q$ no more than
  $\epsilon_b$.
  Suppose that $\epsilon \leq \epsilon_b$ and
  $\mathcal{K}^{(1)} = \{K^{(1)}_{\alpha}: 1 \leq \alpha \leq \frac{(n-1)(n-2)}{2}\}$,
  $\mathcal{K}^{(2)} = \{K^{(2)}_{\alpha}: 1 \leq \alpha \leq \frac{(n-1)(n-2)}{2}\}$, are two normalized set of rotation vector fields, assume that
  \begin{itemize}
    \item $\lambda_1+\lambda_2\geq \delta H$ in $B_{g}(\bar{x},H(\bar{x})^{-1})$,
        where $\lambda_1, \lambda_2$ are the lowest principal curvatures.
    \item $\max_{\alpha}|\left<K^{(i)}_{\alpha},\nu\right>|H\leq \epsilon$
    in  $B_{g}(\bar{x},H(\bar{x})^{-1})$
    \item $\max_{\alpha}|K_{\alpha}^{(j)}|H\leq 5n$ at $\bar{x}$
  \end{itemize}
  for $i=1,2$. Then for $L\geq 1$,
  \begin{align*}
    \inf\limits_{\omega\in O(\frac{(n-1)(n-2)}{2})}\sup\limits_{B_{LH(\bar{x})^{-1}}(\bar{x})}\max_{\alpha}
    |K_{\alpha}^{(1)}-\sum_{\beta = 1}^{\frac{(n-1)(n-2)}{2}}\omega_{\alpha\beta}K^{(2)}_{\beta}|H(\bar{x})
    \leq CL\epsilon
  \end{align*}

  \begin{Rem}
    The condition  $\lambda_1+\lambda_2\geq \delta H$ means that the point has bounded distance to $\{p\}\times\mathbb{R}$ (note that we have normalized the mean curvature),  where $p$ denotes the tip of the Bowl soliton.
  \end{Rem}

  \begin{proof}
    The proof is analogous to \cite{zhu2020so}. Let's make the convention that
    the tip of $\Sigma$ is the origin, the rotation axis is $x_n$, and
    $\Sigma$ encloses the positive part of $x_n$ axis.
    Argue by contradiction, if the assertion is not true, then there exists
    a sequence of points $q_j\in \kappa_j^{-1}\Sigma\times\mathbb{R}$ with
    $H(q_j)=1$ and a sequence of pointed hypersurfaces $(M_j, p_j)$ that are
    $1/j$ close to a geodesic ball $B_{\tilde{g_j}}(q_j,2)$ in
    $\kappa_j^{-1}\Sigma\times\mathbb{R}$, where $\tilde{g_j}$ are the induced
    metric on $\kappa_j^{-1}\Sigma\times\mathbb{R}$. Suppose that $|p_j-q_j|\leq 1/j$.

   Without loss of generality we may assume that
     $\left<q_j,\omega_{n+1}\right>=0$
    where $\omega_{n+1}$ is the unit vector in the $\mathbb{R}$ (splitting)
    direction.

  Further, there exists normalized set of rotation vector fields $\mathcal{K}^{(i,j)}$=$\{K^{(i,j)}_{\alpha}$, $1\leq \alpha\leq \frac{(n-1)(n-2)}{2}\}$   $i=1,2$ and
  $\epsilon_j<1/j$  such that
  \begin{itemize}
    \item $\max_{\alpha}|\left<K^{(i,j)}_{\alpha},\nu\right>|H\leq \epsilon_j$ in $B_{g_j}(p_j,H(p_j)^{-1})\subset M_j$
    \item $\max_{\alpha}|K^{(i,j)_{\alpha}}|H\leq 5$ at $p_j$
  \end{itemize}
  for $i=1,2$, but
  \begin{itemize}
    \item $\inf\limits_{\omega\in O(\frac{(n-1)(n-2)}{2})}\sup\limits_{B_{LH(\bar{x})^{-1}}(\bar{x})}\max_{\alpha}
    |K_{\alpha}^{(1,j)}-\sum_{\beta = 1}^{\frac{(n-1)(n-2)}{2}}\omega_{\alpha\beta}K^{(2,j)}_{\beta}|H(\bar{x})
    \geq j\epsilon_j$
  \end{itemize}

  Now the maximal mean curvature of $\kappa_j^{-1}\Sigma\times\mathbb{R}$ is $\kappa_j$.
  For any $j>2C/\delta$, by the first condition and approximation we know that $\frac{\lambda_1+\lambda_2}{H}\geq \frac{\delta}{2}$ around $q_j$.
  The asymptotic behaviour of the Bowl soliton indicates that $\frac{H(q_j)}{\kappa_j}<C(\delta)$ and $|q_j-\left<q_j,\omega_{n+1}\right>\omega_{n+1}|\kappa_j<C(\delta)$, thus
  $\kappa_j>C(\delta)^{-1}$ and $|q_j|=|q_j-\left<q_j,\omega_{n+1}\right>\omega_{n+1}|<C(\delta)$.

  We can then pass to a subsequence such that $q_j\rightarrow q_{\infty}$ and $\kappa_j\rightarrow\kappa_{\infty}>C(\delta)^{-1}>0$.
  Consequently $\kappa_j^{-1}\Sigma\times\mathbb{R}\rightarrow \kappa_{\infty}^{-1}\Sigma\times\mathbb{R}$ and  $B_{\tilde{g}}(q_j,2)\rightarrow B_{\tilde{g}_{\infty}}(q_{\infty},2)$ smoothly,
   where $B_{\tilde{g}_{\infty}}(q_{\infty},2)$ is the geodesic ball in $\kappa_{\infty}^{-1}\Sigma\times\mathbb{R}$.

  Combing with the assumption that $(M_j,p_j)$ is $1/j$  close to $(B_{\tilde{g}_j}(q_j,2),q_j)$ and $H(q_j)=1$, we have $M_j\rightarrow B_{\tilde{g}_{\infty}}(q_{\infty},2)$ with $p_j\rightarrow q_{\infty}$
  and $H(q_{\infty})=1$.

    We can write $K_{\alpha}^{(i,j)}(x)=S_{(i,j)}J^{(i,j)}_{\alpha}S_{(i,j)}^{-1}(x-b_{(i,j)})$ for some orthonormal basis $\{J^{(i,j)}_{\alpha}, 1\leq \alpha\leq \frac{(n-1)(n-2)}{2}\}$ of
    $so(n-1)$
    and assume that $(b_{(i,j)}-p_j)\perp \bigcap_{\alpha}\ker J^{(i,j)}_{\alpha}S_{(i,j)}^{-1}$. Then
    $|b_j-p_j|=|S_{(i,j)}J^{(i,j)}_{\alpha}S_{(i,j)}^{-1}(p_j-b_j)|\leq C(n)$ for large $j$.

    We can pass to a subsequence such that $S_{(i,j)}$ and $b_{(i,j)}$ converge to $S_{(i,\infty)}, b_{(i,\infty)}$ respectively.
   Consequently $K^{(i,j)}\rightarrow K^{(i,\infty)}_{\alpha} = S_{(i,\infty)}J^{(i,j)}_{\alpha}S_{(i,\infty)}^{-1}(x-b_{(i,\infty)})$ for $i=1,2$.

  The convergence implies that
   \begin{itemize}
     \item $\left<K^{(i,\infty)}_{\alpha},\nu\right> = 0$ in $B_{g_{\infty}}(p_{\infty,}, H_0^{-1})$
     \item $\max_{\alpha}|K^{(i,\infty)}_{\alpha}| H \leq 5n$ at
      $p_{\infty}$
   \end{itemize}
    By Lemma \ref{Lin Alg Cylinder A} we have:
    \begin{flalign*}
      K_{\alpha}^{(i,\infty)}(x) = \sum_{\beta=1}^{\frac{(n-1)(n-2)}{2}}\omega^{(i)}_{\alpha\beta} J_{\beta}x
    \end{flalign*}
    for some fixed orthonormal basis $\{J_{\alpha}, 1\leq \alpha\leq \frac{(n-1)(n-2)}{2}\}$ of $so(n-1)$.

    Arguing exactly as in Lemma \ref{Vector Field closeness on Cylinder} to reach
    a contradiction.
  \end{proof}

\end{Lemma}

\begin{Rem}\label{Alternative Conclusion}
  Under the assumption of Lemma \ref{Vector Field closeness on Cylinder} or Lemma \ref{VF close Bowl times R}, an
  alternative conclusion is also proved:
  if $K^{(1)}_{\alpha} = SJ_{\alpha}S^{-1}(x-b)$, then there
  exists $S', b'$ such that $|S'-S|+|b-b'|\leq C\epsilon$ and
  $\omega\in O(\frac{(n-1)(n-2)}{2})$ such that for each $\alpha$
  \begin{align}\label{V F closeness remark on Cylinder}
    K^{(2)}_{\alpha} = \sum_{\beta=1}^{\frac{(n-1)(n-2)}{2}} \omega_{\alpha\beta} S' J_{\beta}S'^{-1}(x-b')
  \end{align}
\end{Rem}

\begin{Th}\label{Cylindrical improvement}
There exists  constant  $L_0>1$ and  $0<\epsilon_0<1/10$ with the following properties: suppose that $M_t$ is  a mean curvature flow solution, if every point in the parabolic neighborhood $\hat{\mathcal{P}}(\bar{x},\bar{t},L_0,L_0^2)$ is $\epsilon$ symmetric and
$(\epsilon_0,100n^{5/2})$ cylindrical,
 where $0<\epsilon\leq \epsilon_0$, then $(\bar{x},\bar{t})$ is $\frac{\epsilon}{2}$ symmetric.
\end{Th}
\begin{proof}
   We will abbreviate some details that are similar to those in \cite{brendle2018uniqueness}, \cite{brendle2019uniqueness} or \cite{zhu2020so}.

   We assume that $L_0$ is large and $\epsilon_0$ is small depending on $L_0$. The constant $C$ is always assumed to depend on $n$.

  For any space-time point $(y,s)\in \hat{\mathcal{P}}(\bar{x},\bar{t},L_0,L_0^2)$, by assumption there is a normalized set of rotation field $\mathcal{K}^{(y,s)} = \{K_{\alpha}^{(y,s)}, 1\leq\alpha \leq \frac{(n-1)(n-2)}{2}\}$ such that
   $\max_{\alpha}|\left<K_{\alpha}^{(y,s)},\nu\right>|H\leq \epsilon$ and
   $\max_{\alpha}|K_{\alpha}^{(y,s)}|H\leq 5n$ in a parabolic neighbourhood $\hat{\mathcal{P}}(y,s,100n^{5/2},100^2n^5)$.

   Without loss of generality we may assume $\bar{t}=-1$, $H(\bar{x})=\sqrt{\frac{n-2}{2}}$
   and $|\bar{x}-(\sqrt{2(n-2)},0,0,0)|\leq \epsilon_0$.

   Let's set up a reference normalized set of rotation vector fields
   $\mathcal{K}^0 =\{K^0_{\alpha}(x) = J_{\alpha}x, 1\leq\alpha\leq\frac{(n-1)(n-2)}{2}\}$  and let
   $\bar{\mathcal{K}}=\mathcal{K}^{(\bar{x},-1)}$,
   where $\{J_{\alpha}, 1\leq\alpha\leq\frac{(n-1)(n-2)}{2}\}$
   is the orthonormal basis of $so(n-1)$.

   By the cylindrical assumption,
    for each $(y,s)\in\hat{\mathcal{P}}(\bar{x},\bar{t},L_0,L_0^2)$,
     the parabolic neighborhood $\hat{\mathcal{P}}(y,s,100^2n^5,100^2n^5)$ is $C(L_0)\epsilon_0$ close to the shrinking cylinder
   $S^{n-2}_{\sqrt{-2(n-2)t}}\times \mathbb{R}^2$ in $C^{10}$ norm.

    We  use the spherical coordinate $(r\Theta,z_1,z_2)$ on $M_t$, under which  $M_t$ is expressed as a radial graph.
    More precisely,
    the set $$\left\{(r\Theta,z_1,z_2)\big|r=r(\Theta,z_1,z_2), z_1^2+z_2^2\leq \frac{L_0^2}{2},\Theta\in S^{n-2}\right\}$$ is contained in $M_t$.

    Let $0=\lambda_0<\lambda_1\leq\lambda_2\leq ...$ be all eigenvalues
       of the Laplacian on the unit sphere $S^{n-2}$ and let $Y_m$ be the
       eigenfunction corresponding to $\lambda_m$ such that $\{Y_m\}$ forms an
       orthonormal basis with respect to $L^2$ inner product on unit sphere.
       In particular, $\lambda_1=...=\lambda_{n-1}=n-2$, and
       for each $k\leq n-1$ we can choose $Y_k$ to be proportional to $\Theta_k$,  the
       $k$-th coordinate
       function of $\Theta$ (if we consider $\Theta$ as unit
       vector in $\mathbb{R}^{n-1}$).    \\

    \noindent\textbf{Step 1:}
   Using Lemma \ref{Vector Field closeness on Cylinder}
   repeatedly we can conclude that, for any $R>1$ and
   $(x_0, t_0)\in \hat{\mathcal{P}}(\bar{x},\bar{t},L_0,L_0^2)$
   \begin{align*}
      \inf\limits_{\omega\in O(\frac{(n-1)(n-2)}{2})}\sup\limits_{B_{RH(\bar{x})^{-1}}(\bar{x})}\max_{\alpha}
    |\bar{K_{\alpha}}-\sum_{\beta = 1}^{\frac{(n-1)(n-2)}{2}}\omega_{\alpha\beta}K^{(x_0, t_0)}_{\beta}|H(\bar{x})
    \leq C(L_0)R\epsilon
   \end{align*}
   Therefore we can replace $K^{(x_0, t_0)}_{\alpha}$ by
   $\sum_{\beta}\omega_{\alpha\beta}K_{\beta}^{(x_0, t_0)}$ for some
   $\omega\in O(\frac{(n-1)(n-2)}{2})$ depending on $(y,s)$ such that
   \begin{align}\label{Single Vector}
      \sup\limits_{B_{RH(\bar{x})^{-1}}(\bar{x})}\max_{\alpha}
    |\bar{K_{\alpha}}- K^{(x_0, t_0)}_{\alpha}|H(\bar{x})
    \leq C(L_0)R\epsilon
   \end{align}

   Further applying Lemma \ref{Vector Field closeness on Cylinder} for
   $\bar{\mathcal{K}}$ and $\mathcal{K}^0$ with $\epsilon_0$, we can find $ \omega\in O(\frac{(n-1)(n-2)}{2})$ such that:
   \begin{align}
      \sup\limits_{B_{RH(\bar{x})^{-1}}(\bar{x})}\max_{\alpha}
    |\bar{K}_{\alpha} - \sum_{\beta = 1}^{\frac{(n-1)(n-2)}{2}}\omega_{\alpha\beta}K^{0}_{\beta}|H(\bar{x})
    \leq CR\epsilon_0
   \end{align}

   By Remark \ref{Alternative Conclusion}, the axis of $\bar{K}_{\alpha}$ and the axis of the approximating cylinder differs by at most $CR\epsilon_0$, therefore by rotating the cylinder we
   may assume that $\bar{\mathcal{K}}=\mathcal{K}^0$.

   Let's remark that since $|\sum_{\beta = 1}^{\frac{(n-1)(n-2)}{2}}\omega_{\alpha\beta}K^{0}_{\beta}|H \equiv n-2$ on $S^{n-2}\times \mathbb{R}^2$, by approximation the condition $|K_{\alpha}|H\leq 5n$ is always satisfied throughout the proof. \\

   \noindent\textbf{Step 2:} Let's fix an $\alpha$. Given any space time point
   $(x_0, t_0)\in \hat{\mathcal{P}}(\bar{x},\bar{t},L_0,L_0^2)$, there exists
   constants $a_i, b_i, c_i$ (depending on the choice of $(x_0, t_0)$) such that
   \begin{align*}
     |a_1| + ...+|a_{n-1}|\leq &C(L_0)\epsilon \\
     |b_1| + ...+|b_{n-1}|\leq &C(L_0)\epsilon \\
     |c_1| + ...+|c_{n-1}| \leq &C(L_0)\epsilon
   \end{align*}
   and
   \begin{align*}
     |\left<\bar{K}_{\alpha}-K_{\alpha}^{(x_0, t_0)},\nu \right>&-(a_1Y_1+...+a_{n-1}Y_{n-1})\\
   &-(b_1Y_1+...+b_{n-1}Y_{n-1})z_1 \\
   &-(c_1Y_1+...+c_{n-1}Y_{n-1})z_2|\leq C(L_0)\epsilon_0\epsilon
   \end{align*}
   holds in $\hat{\mathcal{P}}(x_0, t_0, 10,100)$.

   Consequently,  function $u=\left<\bar{K_{\alpha}},\nu\right>$ satisfies
   \begin{align}\label{u Bound}
     |u&-(a_1Y_1+...+a_{n-1}Y_{n-1})\notag\\
   &-(b_1Y_1+...+b_{n-1}Y_{n-1})z_1 \notag\\
   &-(c_1Y_1+...+c_{n-1}Y_{n-1})z_2|\leq C(L_0)\epsilon_0\epsilon + C(-t_0)^{1/2}\epsilon
   \end{align}
   in $\hat{\mathcal{P}}(x_0, t_0, 10,100)$.

   Function $u$ satisfies the Jacobi equation:
   \begin{align}\label{Jacobi Equation}
     \frac{\partial u}{\partial t} = \Delta u + |A|^2u
   \end{align}
   using parabolic interior estimate we obtain $|\nabla u| + |\nabla^2 u|\leq C(L_0)\epsilon$ in $\hat{\mathcal{P}}(\bar{x},-1, \frac{L_0}{\sqrt{2}}, \frac{L_0^2}{2})$.

   By approximation,
    \begin{align}\label{Nonhom Jacobi}
       |\frac{\partial u}{\partial t}-(\frac{\partial^2 u}{\partial z_1^2}+\frac{\partial^2 u}{\partial z_2^2}+\frac{\Delta_{S^{n-2}}u}{-2(n-2)t}+\frac{1}{-2t}u)|\leq C(L_0)\epsilon\epsilon_0
     \end{align}
     holds for $z_1^2+z_2^2 \leq \frac{L_0^2}{4}$ and $-\frac{L_0^2}{4}\leq t\leq -1$.

      Define

      $\Omega_l(\bar{z}_1,\bar{z}_2) =\{(z_1,z_2)\in\mathbb{R}^2| \ |z_i-\bar{z}_i|\leq l\}$

      $\Omega_l=\Omega_l(0, 0)$

       $\Gamma_l=\{(z_1,z_2)\in \Omega_l,\theta\in [0,2\pi],t\in [-l^2,-1] \big| t=-l^2 \text{ or } |z_1|=l \text{ or } |z_2|=l\}$

      Let $\tilde{u}$ solves the homogenous equation
      \begin{align*}
       \frac{\partial u}{\partial t}=\frac{\partial^2 u}{\partial z_1^2}+\frac{\partial^2 u}{\partial z_2^2}+\frac{\Delta_{S^{n-2}}u}{-2(n-2)t}+\frac{1}{-2t}u
      \end{align*}
       in $\Omega_{L_0/4}\times [0,2\pi]\times [-L_0^2/16,-1]$ and satisfies the boundary condition $\tilde{u}=u$ on $\Gamma_{L_0/4}$.
       By (\ref{Nonhom Jacobi}) and maximal principle
       \begin{align}\label{diff tildeu and u}
         |u-\tilde{u}|\leq C(L_0)\epsilon\epsilon_0
       \end{align}

       Next, we compute Fourier coefficients of $\tilde{u}$  and analysis each of them.
       \begin{align*}
         v_m=\int_{S^{n-2}}\tilde{u}(z,t,\Theta)Y_m(\Theta) d\Theta
       \end{align*}
        where $z=(z_1,z_2)$. $v_m$ satisfies the following equation:
        \begin{align*}
       \frac{\partial v_m}{\partial t}=\frac{\partial^2 v_m}{\partial z_1^2}+\frac{\partial^2 v_m}{\partial z_2^2}+\frac{n-2-\lambda_m}{-2(n-2)t}v_m
      \end{align*}
      Therefore

       Let $\hat{v}_m=v_m(-t)^{\frac{n-2-\lambda_m}{2(n-2)}}$.
       Then $\hat{v}_m$ satisfies the linear heat equation
       \begin{align*}
         \frac{\partial \hat{v}_m}{\partial t}=\frac{\partial^2 \hat{v}_m}{\partial z_1^2}+\frac{\partial^2 \hat{v}_m}{\partial z_2^2}
       \end{align*}

       \noindent\textit{Case 1: } $m\geq n$

       In this case $\lambda_m\geq 2(n-1)$. Taking the $L^2$ inner product with
       $Y_m$ on both sides of (\ref{u Bound}) and multiplying by
       $(-t)^{\frac{n-2-\lambda_m}{2(n-2)}}$   we obtain that:
       \begin{align*}
         |\hat{v}_m|\leq (C(L_0)\epsilon\epsilon_0+C\epsilon)(-t)^{1-\frac{\lambda_m}{2(n-2)}}
       \end{align*}
        in $\Omega_{L_0/4}\times [-L_0^2/16,-1]$

       The heat kernel with Dirichlet Boundary for $\Omega_{L_0/4}$ is
       \begin{align*}
          K_t&(x,y)=-\frac{1}{4\pi t}\sum_{\delta_i\in \{\pm 1\}, k_i\in \mathbb{Z}} (-1)^{-(\delta_1+\delta_2)/2}\cdot\\ &\exp\left({-\frac{\left|(x_1,x_2)-(\delta_1 y_1, \delta_2 y_2)-(1-\delta_1,1-\delta_2)\frac{L_0}{4}+(4k_1,4k_2)\frac{L_0}{4}\right|^2}{4t}}\right)
       \end{align*}
       where $x=(x_1,x_2), y=(y_1, y_2)$ are in $\Omega_{L_0/4}$.

       $K_t$ satisfies
       \begin{enumerate}
         \item $\partial_tK_t=\Delta_x K_t$ in  $\Omega_{L_0/4}\times (0,\infty) $
         \item $K_t$ is symmetric in $x$ and $y$
         \item $\lim_{t\rightarrow 0}K_t(x,y)=\delta(x-y)$ in the distribution sense.
         \item $K_t(x,y)=0$ \ for $y\in\partial\Omega_{L_0/4}$ and $t>0$
       \end{enumerate}
       The solution formula is
       \begin{align*}
         \hat{v}_m(x,t)= &\int_{\Omega_{L_0/4}}K_{t+L_0^2/16}(x,y)\hat{v}_m(y,-\frac{L_0^2}{16})dy \\ -&\int_{-L_0^2/16}^{t}\int_{\partial\Omega_{L_0/4}}\partial_{\nu_y}K_{t-\tau}(x,y)\hat{v}_m(y,\tau)dy \ d\tau
       \end{align*}

       Now we let $(x,t)\in \Omega_{L_0/100}\times [-L_0^2/100^2,-1]$ ($L_0$ is large enough) and $-L_0^2/16\leq\tau<t$.
       For all such $(x,t)$ and $\tau$ we have the following heat kernel estimate (see Appendix \ref{Appendix HKEST} for details): 

         \begin{align}\label{hkest1}
           \int_{\Omega_{L_0/4}}|K_{t+L_0^2/16}(x,y)|dy\leq C
         \end{align}

        \begin{align}\label{hkest2}
          \int_{\partial{\Omega_{L_0/4}}}|\partial_{\nu_y}K_{t-\tau}(x,y)|dy\leq \frac{CL_0^2}{(t-\tau)^2}e^{-\frac{L_0^2}{1000(t-\tau)}}
        \end{align}

       Now we can estimate $\hat{v}_m(\bar{x},\bar{t})$ for $(\bar{x},\bar{t})\in \Omega_{L_0/100}\times [-\frac{L_0^2}{100^2},-1]$.
       \begin{align*}
         |\hat{v}_m(x,t)|&\leq (C(L_0)\epsilon_0\epsilon+C\epsilon)\Big(\frac{L_0^2}{16}\Big)^{1-\frac{\lambda_m}{2(n-2)}} \\
         &+(C(L_0)\epsilon_0\epsilon+C\epsilon)\int_{-L_0^2/16}^{t}\frac{CL_0^2}{(t-\tau)^2}e^{-\frac{L_0^2}{1000(t-\tau)}}(-\tau)^{1-\frac{\lambda_m}{2(n-2)}}d\tau\\
       \end{align*}

       For $L_0$ large and $-200^2n^5\leq t\leq -1$
       \begin{align*}
         \frac{CL_0^2}{(t-\tau)^2}e^{-\frac{L_0^2}{2000(t-\tau)}}(-\tau)^{1-\frac{1}{2(n-2)}}\leq CL_0^{-\frac{1}{n-2}}
       \end{align*}
       whenever $\tau<t$.

       Therefore
       \begin{align}\label{estimate large m}
         |\hat{v}_m(x,t)|\leq (C(L_0)\epsilon_0\epsilon+C\epsilon)&\Big(\frac{L_0^2}{16}\Big)^{1-\frac{\lambda_m}{2(n-2)}} \notag \\ \notag
         +(C(L_0)\epsilon_0\epsilon+C\epsilon)&\int_{-\frac{L_0^2}{16}}^{t}CL_0^{-\frac{1}{n-2}} e^{-\frac{L_0^2}{2000(t-\tau)}}(-\tau)^{\frac{1-\lambda_m}{2(n-2)}}d\tau\\\notag
         \leq(C(L_0)\epsilon_0\epsilon+C\epsilon)&\Big[\Big(\frac{L_0}{4}\Big)^{2-\frac{\lambda_m}{n-2}}+\int_{(1+\frac{1}{\sqrt{\lambda_m}})t}^{t}C e^{-\frac{L_0^2}{2000(t-\tau)}}d\tau\\  \notag &+\int_{-\infty}^{(1+\frac{1}{\sqrt{\lambda_m}})t}CL_0^{-\frac{1}{n-2}}(-\tau)^{\frac{1-\lambda_m}{2(n-2)}}d\tau \Big] \\ \notag
         \leq(C(L_0)\epsilon_0\epsilon+C\epsilon)&\Big[\Big(\frac{L_0}{4}\Big)^{2-\frac{\lambda_m}{n-2}}+\frac{C}{\sqrt{\lambda_m}}e^{-\frac{\sqrt{\lambda_m}L_0^2}{2000}}\\ &+CL_0^{-\frac{1}{n-2}}\left(1+\frac{1}{\sqrt{\lambda_m}}\right)^{\frac{2n-3-\lambda_m}{2(n-2)}} \Big]
       \end{align}

       Note that $2-\frac{\lambda_m}{n-2}\leq \frac{-2}{n-2}$ and $\frac{2n-3-\lambda_m}{2(n-2)}\leq -\frac{1}{2(n-2)}$. Morevoer,
       The eigenvalues of Laplacian on $S^{n-2}$ are $l(l+n-3)$ for integer
       $l\geq 1$ with multiplicity $N_l={n+l-2 \choose n-2}-{n+l-4 \choose n-2}$.
       For large $l$ we have $l(l+n-3)\leq Cl^2$ and $N_l\leq Cl^{n-2}$.
       Putting these facts into (\ref{estimate large m}) and summing over
       $m\geq n$ we obtain:
       \begin{align*}
         \sum_{m=n}^{\infty} |\tilde{v}_m|\leq  C(L_0)\epsilon_0\epsilon+CL_0^{-\frac{1}{n-2}}\epsilon
       \end{align*}
       for $(x,t)\in\Omega_{200n^{5/2}}\times[-200^2n^5,-1]$ \\

       \noindent\textit{Case 2: } $1\leq m \leq n-1$

       In this case $v_m = \hat{v}_m$. Recall that they satisfies the heat equation:
       \begin{align*}
         \frac{\partial v_m}{\partial t}=\frac{\partial^2 v_m}{\partial z_1^2}+\frac{\partial^2 v_m}{\partial z_2^2}  \ \  \text{  in } \Omega_{L_0/4}\times [-L_0^2/16,-1]
       \end{align*}
       For each $(X_0,t_0)=(x_1,x_2,t_0)\in\Omega_{L_0/4} \times [-L_0^2/16,-1]$,  taking the Fourier coefficient of $Y_m$ of (\ref{u Bound}), we can find $a_m, b_m, c_m $ satisfying $|a_m|+|b_m|+|c_m|\leq C(L_0)\epsilon$ and
        $$|v_m - (a_m + b_m z_1 + c_m z_2)|\leq (C(L_0)\epsilon\epsilon_0+C\epsilon)(-t)^{1/2}$$
       in $\Omega_{(-t_0)^{\frac{1}{2}}}(x_1,x_2)\times[0,2\pi]\times[2t_0,t_0]$.

       The linear term satisfies the heat equation trivially, so the parabolic interior gradient estimate implies that $|\frac{\partial^2 v_m}{\partial z_i\partial z_j}|\leq(C(L_0)\epsilon\epsilon_0+C\epsilon)(-t)^{-1/2}$ for each pair of $1\leq i,j\leq2$ in  $\Omega_{L_0/8}\times[-L_0^2/64,-1]$.

       Note that $\frac{\partial^2 v_m}{\partial z_i\partial z_j}$  also satisfies the heat equation in $\mathbb{R}^2$,
       we can apply the same argument as in $m\geq n$ case  to $\frac{\partial^2 v_m}{\partial z_i\partial z_j}$  to obtain that:
       \begin{align*}
         |\frac{\partial^2 v_m}{\partial z_i\partial z_j}|&\leq (C(L_0)\epsilon_0\epsilon+C\epsilon)\Big(\frac{L_0^2}{64}\Big)^{-\frac{1}{2}} \\
         &+(C(L_0)\epsilon_0\epsilon+C\epsilon)\int_{-L_0^2/64}^{t}\frac{CL_0^2}{(t-\tau)^2}e^{-\frac{L_0^2}{1000(t-\tau)}}(-\tau)^{-\frac{1}{2}}d\tau\\
       \end{align*}
       for each $1\leq i, j\leq 2$.

       For $L_0$ large, $-200^2n^5\leq t\leq -1$.
       \begin{align*}
         \frac{CL_0^2}{(t-\tau)^2}e^{-\frac{L_0^2}{1000(t-\tau)}}\leq CL_0^{-2}
       \end{align*}
       whenever $\tau<t$.

       Then we have:

       \begin{align}
         \sum_{i,j=1,2}|\frac{\partial^2 v_m}{\partial z_i\partial z_j}|\leq C(L_0)\epsilon_0\epsilon+CL_0^{-1}\epsilon
       \end{align}
      for $(x,t)\in\Omega_{200n^{5/2}}\times[-200^2n^5,-1]$ and any $1\leq m \leq n-1$.
      This means that we can find real numbers $A_m, B_m, C_m$ such that
      \begin{align}
        |v_m-(A_m+B_m z_1+C_m z_2)|\leq C(L_0)\epsilon_0\epsilon+CL_0^{-1}\epsilon
      \end{align}
      and $|A_m|+|B_m|+|C_m|\leq C(L_0)\epsilon$. \\

      \noindent\textit{Case 3: } $m=0$

      Under the polar coordinate on $M_t$, the normal vector $\nu$ satisfies:
      \begin{align*}
        \nu\sqrt{1+\frac{|\nabla^{S^{n-2}} r|}{r}^2+(\frac{\partial u}{\partial z_1})^2+(\frac{\partial u}{\partial z_2})^2}=(\Theta,-\frac{\partial r}{\partial z_1},-\frac{\partial r}{\partial z_2})+(-\frac{\nabla^{S^{n-2}} r}{r}, 0, 0)
      \end{align*}
      Since $J_{\alpha}\Theta$ is divergence free on $S^{n-2}$, we have
      \begin{align*}
        &\int_{S^{n-2}}u\sqrt{1+\frac{|\nabla^{S^{n-2}} r|}{r}^2+(\frac{\partial u}{\partial z_1})^2+(\frac{\partial u}{\partial z_2})^2}d\Theta \\
        =&\int_{S^{n-2}}\left<J_{\alpha}\Theta, -\nabla^{S^{n-2}}r \right>_{S^{n-2}}d\Theta \\
        =&\int_{S^{n-2}} \text{div}_{S^{n-2}}(rJ_{\alpha}\Theta)d\Theta = 0
      \end{align*}
      Since $|u|\leq C(L_0)\epsilon$ and $|\frac{\nabla^{S^{n-2}} r}{r}|+|\frac{\partial u}{\partial z_1}|
      +|\frac{\partial u}{\partial z_2}|\leq C\epsilon_0$,   we  obtain that
      \begin{align}\label{mode0estimate}
        \left|\int_{0}^{2\pi}u(\Theta,z_1,z_2)d\Theta\right|\leq C(L_0)\epsilon_0\epsilon
      \end{align}
      in $\Omega_{200n^{5/2}}\times[-200^2n^5,-1]$.

      We then conclude that $|\hat{v}_0|\leq C(L_0)\epsilon_0\epsilon$ in $\Omega_{200n^{5/2}}\times[-200^2n^5,-1]$ \\

      \noindent\textbf{Step 3: }Combing the analysis for all $m\geq 0$ and the fact that $|u-\tilde{u}|\leq C(L_0)\epsilon_0\epsilon$,
       we conclude that
       there exists  constants $A_{\alpha,i}, B_{\alpha,i}, C_{\alpha,i}$ for
       $1\leq \alpha \leq \frac{(n-1)(n-2)}{2}$ and  $1\leq i\leq n-1$ such that
       \begin{align*}
         |A_{\alpha,1}| + ...+|A_{\alpha,n-1}|\leq &C(L_0)\epsilon \\
         |B_{\alpha,1}| + ...+|B_{\alpha,n-1}|\leq &C(L_0)\epsilon \\
         |C_{\alpha,1}| + ...+|C_{\alpha,n-1}| \leq &C(L_0)\epsilon
       \end{align*}
      and
      \begin{align}\label{untuned u bound}
      |\left<\bar{K}_{\alpha},\nu\right>&-(A_{\alpha,1}Y_1+...+A_{\alpha,n-1}Y_{n-1})\notag\\
       &-(B_{\alpha,1}Y_1+...+B_{\alpha,n-1}Y_{n-1})z_1 \notag\\
       &-(C_{\alpha,1}Y_1+...+C_{\alpha,n-1}Y_{n-1})z_2|\leq C(L_0)\epsilon_0\epsilon+CL_0^{-\frac{1}{n-2}}\epsilon
       \end{align}
      in $S^{n-2}\times\Omega_{200n^{5/2}}\times[-200^2n^5,-1]$.

      For each $i\in\{1,2,...,n-1\}$, define
      \begin{align*}
        F_i(z_1,z_2) = \int_{S^{n-2}}r(\Theta,z_1,z_2)Y_i(\Theta)d\Theta
      \end{align*}
      (recall that $r$ is the radius function under the polar coordinate, which is approximately a constant)
      We compute
      \begin{align}\label{L2 inner product 1}
        \int_{S^{n-2}}\left<\bar{K}_{\alpha},\nu\right>Y_id\Theta= &\int_{S^{n-2}}-\text{div}_{S^{n-2}}(rJ_{\alpha}\Theta)Y_i \notag\\
        =&\int_{S^{n-2}} -\text{div}_{S^{n-2}}(rJ_{\alpha}\Theta Y_i)+\left<\nabla^{S^{n-2}}Y_i,rJ_{\alpha}\Theta\right> \notag\\
        =& \sum_{j=1}^{n-1} \int_{S^{n-2}} J_{\alpha,ij}Y_j rd\Theta \notag\\
        =& \sum_{j=1}^{n-1}  J_{\alpha,ij}F_j(z_1,z_2)
      \end{align}

      On the other hand, if we take $L^2$ inner product with $Y_i$ on both sides
      of (\ref{untuned u bound}) we have
      \begin{align}\label{L2 prod 2}
        \left|\int_{S^{n-2}}\left<\bar{K}_{\alpha},\nu\right>Y_id\Theta - (A_{\alpha, i} + B_{\alpha,i}z_1 + C_{\alpha,i}z_2 )\right|\leq C(L_0)\epsilon_0\epsilon+CL_0^{-\frac{1}{n-2}}\epsilon
      \end{align}

      Let $b\in \mathbb{R}^{n+1}$ and $P\in so(n-1)^{\perp}\subset so(n+1)$ satisfy
      \begin{align}\label{define b p}
        b_i =& F_i(0,0)\notag\\
        P_{n,i}= &-\frac{1}{2}(F_i(1,0)-F_i(-1,0))\\
        P_{n+1,i}=&-\frac{1}{2}(F_i(0,1)-F_i(0,-1))\notag
      \end{align}
      for each $i\in\{1,2,...,n-1\}$. Hence
       \begin{align}\label{Lie Bracket P J}
        [P,J_{\alpha}]_{n,i} =& -\sum_{j=1}^{n-1} \frac{1}{2}(F_j(1,0)-F_j(-1,0))J_{\alpha,ji} \notag\\
        [P,J_{\alpha}]_{n+1,i} =& -\sum_{j=1}^{n-1} \frac{1}{2}(F_j(0,1)-F_j(0,-1))J_{\alpha,ji}\\
        |P|+|b|\leq &C(L_0)\epsilon \notag
      \end{align}
      for each $i$.
      Now we combine the computations   (\ref{L2 inner product 1}) (\ref{L2 prod 2}) and ($\ref{Lie Bracket P J}$)
     to conclude that, for each $1\leq \alpha\leq \frac{(n-1)(n-2)}{2}$ and $1\leq i\leq n-1$:
     \begin{align*}
       |A_{\alpha,i} - (J_{\alpha}b)_i| \leq C(L_0)\epsilon_0\epsilon+CL_0^{-\frac{1}{n-2}}\epsilon\\
       |B_{\alpha,i} -[P,J_{\alpha}]_{n,i}|\leq C(L_0)\epsilon_0\epsilon+CL_0^{-\frac{1}{n-2}}\epsilon \\
       |C_{\alpha,i} -[P,J_{\alpha}]_{n+1,i}|\leq C(L_0)\epsilon_0\epsilon+CL_0^{-\frac{1}{n-2}}\epsilon
     \end{align*}

      Now we let $S=\exp(P)$ and let
      \begin{align*}
        \tilde{K}_{\alpha} = SJ_{\alpha}S^{-1}(x-b)
      \end{align*}

      Then it's easy to compute that on the cylinder $S^{n-2}\times\mathbb{R}^2$
      (of any radius) we have
      \begin{align}\label{diff tuning}
        |\left<\bar{K}_{\alpha}-\tilde{K}_{\alpha},\nu\right> - &\left<[P,J_{\alpha}]x - J_{\alpha}b,\nu\right> | \leq  C|P|^2+C|P||b|+C(L_0)\epsilon_0\epsilon+CL_0^{-\frac{1}{n-2}}\epsilon \notag\\
        \Rightarrow|\left<\bar{K}_{\alpha}-\tilde{K}_{\alpha},\nu\right> - &(A_{\alpha,1}Y_1+...+A_{\alpha,n-1}Y_{n-1})\\
       -&(B_{\alpha,1}Y_1+...+B_{\alpha,n-1}Y_{n-1})z_1 \notag\\
       -&(C_{\alpha,1}Y_1+...+C_{\alpha,n-1}Y_{n-1})z_2| \leq C(L_0)\epsilon_0\epsilon+CL_0^{-\frac{1}{n-2}}\epsilon\notag\\ \notag
      \end{align}
      for each $1\leq\alpha \leq \frac{(n-1)(n-2)}{2}$.

      By approximation and (\ref{diff tuning}), (\ref{untuned u bound}) we conclude that,
      \begin{align*}
        |\left<\tilde{K}_{\alpha},\nu\right>|H \leq C(L_0)\epsilon_0\epsilon + CL_0^{-\frac{1}{n-2}}\epsilon
      \end{align*}
      in $\hat{\mathcal{P}}(\bar{x},\bar{t},100n^{5/2},100^2n^5)$,
      for each $1\leq\alpha \leq \frac{(n-1)(n-2)}{2}$.

      We choose $L_0$ large enough and then choose $\epsilon_0$ small enough depending on $L_0$, then $(\bar{x},\bar{t})$ is $\frac{\epsilon}{2}$
      symmetric.
\end{proof}

\begin{Th}\label{bowl x R improvement}
There exists a constant  $L_1\gg L_0$ and  $0<\epsilon_1\ll \epsilon_0$ such that: for a mean curvature flow solution $M_t$ and a space-time point $(\bar{x},\bar{t})$, if $\hat{\mathcal{P}}(\bar{x},\bar{t},L_1,L_1^2)$ is $\epsilon_1$ close to a piece of $\text{Bowl}^{n-1} \times \mathbb{R}$ in $C^{10}$ norm after the parabolic rescaling (which makes $H(\bar{x},\bar{t})=1$)
and every points in $\hat{\mathcal{P}}(\bar{x},\bar{t},L_1,L_1^2)$ are $\epsilon$ symmetric with $0<\epsilon\leq \epsilon_1$, then $(\bar{x},\bar{t})$ is $\frac{\epsilon}{2}$ symmetric.
\end{Th}

\begin{proof}
  Throughout the proof, $L_1$ is always assumed to be large enough depending only on $L_0,\epsilon_0, n$ and $\epsilon_1$ is assumed to be small enough depending  on $L_1, L_0, \epsilon_0, n$.  $C$ denotes the constant depending only on $n, L_0, \epsilon_0$.

  We denote by $\Sigma$ the standard Bowl soliton in $\mathbb{R}^{n}$.
  That means the tip of $\Sigma$ is the origin, the mean curvature at the tip is $1$ and the rotation axis  is $x_n$, also it enclose the positive part of $x_n$ axis. Let $\omega_n$ be the unit vector in the $x_n$ axis that coincide with the inward normal vector of $\Sigma$ at the origin.
  We write $\Sigma_t=\Sigma+\omega_n t$, then $\Sigma_t$ is the translating mean curvature flow solution and $\Sigma=\Sigma_{0}$.

  After rescaling we may assume that $H(\bar{x},\bar{t})=1$ and $\bar{t}=0$ .
  Moreover, there exists a scaling factor $\kappa>0$ such that the parabolic neighborhood $\hat{\mathcal{P}}(\bar{x},0,L_1,L_1^2)$ can be approximated by the family of translating $\kappa^{-1}\Sigma_t\times \mathbb{R}$ with an error that is $\epsilon_1$ small in $C^{10}$ norm.

  By the above setup, the maximal mean curvature of $\kappa^{-1}\Sigma_t$ is $\kappa$. Let $p_t=\omega_n t\in\kappa^{-1}\Sigma_t$ be the tip.
  Let the straight line $l_t=\{p_t\}\times\mathbb{R}\subset \kappa^{-1}\Sigma_t\times
  \mathbb{R}$. The splitting direction $\mathbb{R}$ is assumed to be $x_{n+1}$ axis.

  By our normalization $H(\bar{x},-1)=1$, the maximality of $\kappa$ together with the approximation ensures that $\kappa\geq 1-C\epsilon_1$.

   Define $d(\bar{x},l_t)=\min_{a\in l_t}|\bar{x}-a|$ to be the Euclidean distance between $\bar{x}$ and the line $l_t$.

   We divide the proof into several steps, the first step deals with the case that $\bar{x}$ is far away from $l_{-1}$ in which we can apply cylindrical improvement. The remaining steps deal with the case that $\bar{x}$ is not too far from $l_{-1}$. \\

   \noindent\textbf{Step 1:} Using the structure of the Bowl soliton and approximation, we can find a large constant $\Lambda_{\star}$ that depends
   only on $L_0, \epsilon_0, n$ such that, if $d(\bar{x},l_{0})\geq\Lambda_{\star}$,
   then for every point $(y,s)\in\hat{\mathcal{P}}(\bar{x},0,L_0,L_0^2)$
   \begin{align*}
     \hat{\mathcal{P}}(y,s,100^2n^5,100^2n^5)\subset\hat{\mathcal{P}}(\bar{x},0,L_1,L_1^2)
   \end{align*} and $(y,s)$ is $(\epsilon_0,100n^{5/2})$ cylindrical.

    Therefore we can apply Theorem \ref{Cylindrical improvement} to conclude that
    $(\bar{x},\bar{t})$ is  $\frac{\epsilon}{2}$ symmetric and we are done.\\

   In the following steps we assume that $d(\bar{x},l_{-1})\leq\Lambda_{\star}$\\

  \noindent\textbf{Step 2:}
  By the asymptotic behaviour of $\Sigma_t$, we have the scale invariant identity:
  $$\frac{\kappa}{H(p)}=f(H(p)d(p,l_t))$$ where $p\in \kappa^{-1}\Sigma_{t}\times\mathbb{R}$ and $f$ is a continuous increasing function such that $f(0)=1$.
  In fact $f(x)=O(x)$ as $x\rightarrow+\infty$.

  By approximation we know that
  \begin{align}\label{upperlowerboundkappa}
     \frac{1}{2}<\kappa<C
  \end{align}
  This means $\kappa^{-1}\Sigma_{t}\times\mathbb{R}$ is equivalent to the standard Bowl$\times\mathbb{R}$ up to a scaling factor depending on $L_0,\epsilon_0,n$. \\

   We define a series of set $\text{Int}_j(\hat{\mathcal{P}}(\bar{x},0,L_1,L_1^2))$ inductively:
   \begin{align*}
   &\text{Int}_{0}(\hat{\mathcal{P}}(\bar{x},0,L_1,L_1^2))=\hat{\mathcal{P}}(\bar{x},0,L_1,L_1^2)\\
  &(y,s) \in \text{Int}_j(\hat{\mathcal{P}}(\bar{x},0,L_1,L_1^2))\\
\Leftrightarrow &
  \hat{\mathcal{P}}(y,s,(10n)^6L_0,(10n)^{12}L_0^2)\subset \text{Int}_{j-1}(\hat{\mathcal{P}}(\bar{x},0,L_1,L_1^2))
  \end{align*}
  We  abbreviate $\text{Int}_j(\hat{\mathcal{P}}(\bar{x},0,L_1,L_1^2))$
  as $\text{Int}(j)$. It's clear that Int($j$) is a decreasing set that are all contained in $\hat{\mathcal{P}}(\bar{x},0,L_1,L_1^2)$. \\

    \noindent\textit{Claim: }  there exists $\Lambda_1\gg \Lambda_{\star}$ depending only on $\epsilon_0, L_0$,
    such that for any point $(x,t)\in \text{Int}_j(\hat{\mathcal{P}}(\bar{x},0,L_1,L_1^2))$,
    if $d(x,l_t)\geq 2^{\frac{j}{100}}\Lambda_1$, then $(x,t)$ is $2^{-j}\epsilon$ symmetric, for $j=0,1,2,...$\\

   By the assumption the case $j=0$ is automatically true. Now suppose that the statement is true for $j-1$.

  By the structure of the Bowl solition and approximation, we can find constant
  $\Lambda_1$ that depends only on $L_0,\epsilon_0,n$ such that
   \begin{align}\label{bigdistimpliescylindrical}
     &\text{If }  d(x,l_t)\geq \Lambda_{1},
    \text{then every point }(y,s)\in\hat{\mathcal{P}}(x,t,L_0,L_0^2) \text{ is } (\epsilon_0,100n^{5/2})\\ &  \text{ cylindrical}\text{ and }  d(x,l_t)H(x)\geq 1000L_0. \notag\\
    &\text{ Moreover } \hat{\mathcal{P}}(y,s,100^2n^5,100^2n^5)\subset\hat{\mathcal{P}}(x,t,(10n)^6L_0,(10n)^{12}L_0^2)\notag
   \end{align}

   \noindent\textit{Remark:} We need to rescale the picture before checking $(\epsilon_0,100n^{5/2})$ cylindrical condition. The scaling factor depends only on $L_1$. Thus it can be conquered by choosing $\epsilon_1$ small enough, since $\epsilon_1$ is chosen after $L_1$. \\

    By approximation and structure of the Bowl soliton, we know that if $(x,t)\in\hat{\mathcal{P}}(\bar{x},0,L_1,L_1^2)$ with $d(x,l_t)\geq 1 \geq \frac{1}{2}\kappa^{-1}$, then
    \begin{itemize}
      \item $\frac{\left<\nu, x-x'\right>}{|x-x'|}\geq -C\epsilon_1$
      \item $\frac{\left<\omega_n, x-x'\right>}{|x-x'|}\geq C(n)^{-1}$
    \end{itemize}
    where $x'$ is a point on $p_t$ such that $|x-x'|=d(x,p_t)$.
    Let $x(t)$ be the trajectory of $x$ under mean curvature flow, such that
    $x(t_0)=x_0$ and $d(x_0,p_{t_0})> 1$, then for $t\leq t_0$
    \begin{align}\label{distancetotipdecreasing}
      \frac{d}{dt}d(x(t),l_t)^2&= \left<x-x',H(x)\nu(x)-\kappa\omega_n\right>\notag\\
      &\leq|x-x'|(C\epsilon_1 - C(n)^{-1}) < 0
    \end{align}
    Therefore $x(t)$ increases as $t$ decreases, as long as $x(t)\in\hat{\mathcal{P}}(\bar{x},0,L_1,L_1^2)$.

    Now suppose that $(x,t)\in\text{Int}(j)$ and $d(x,l_t)\geq 2^{\frac{j}{100}}\Lambda_1.$
    For any point $(\tilde{x},t)\in\hat{\mathcal{P}}(x,t,L_0,L_0^2)$, by (\ref{bigdistimpliescylindrical})
    \begin{align}\label{distance increasing1}
      d(\tilde{x},l_t)&\geq d(x,l_t)-d(x,\tilde{x})\notag\\
      &\geq d(x,l_t)-L_0H(x)^{-1}\notag\\
      &\geq (1-1/1000)d(x,l_t)\notag\\
      &\geq (1-1/1000) 2^{\frac{j}{100}}\Lambda_1>2^{\frac{j-1}{100}}\Lambda_1
    \end{align}
    Additionally, every point $(y,s)\in\hat{\mathcal{P}}(x,t,L_0,L_0^2)$ must be in Int($j-1$), hence by (\ref{distancetotipdecreasing}) and (\ref{distance increasing1}) $d(y,l_s)>2^{\frac{j-1}{100}}\Lambda_1$ and therefore is $2^{-j+1}\epsilon$ symmetric by induction hypothesis.

    Together with (\ref{bigdistimpliescylindrical}), the condition of the cylindrical improvement are satisfied. Hence $(x,t)$ is $2^{-j}\epsilon$ symmetric and the conclusion is true for $j$, the Claim follow from induction.\\

   Finally, we describe the size of Int($j$) in terms of $n,L_0$.
    Recall that $d(\bar{x},l_{0})\leq \Lambda_{\star}$. By approximation we have:

    \begin{align*}
      &\hat{\mathcal{P}}(\bar{x},0,(10n)^{-8}L_0^{-1}\hat{L},(10n)^{-16}L_0^{-2}\hat{L}^2)
    \subset \text{Int}_1\hat{\mathcal{P}}(\bar{x},0,\hat{L},\hat{L}^2)
    \end{align*}
    whenever $(10n)^{-16} L_0^{-2}\hat{L}>\Lambda_1$ and $\hat{L}<L_1$.

    By  induction on $j$ we obtain:
    \begin{align}\label{size of Int 2}
      \hat{\mathcal{P}}(\bar{x},0,(10n)^{-8j} L_0^{-j}L_1,(10n)^{-16j} L_0^{-2j}L_1)\subset\text{Int}(j)
    \end{align}
    whenever $(10n)^{-8j-8}L_0^{-j-1}L_1>\Lambda_1$. \\

    \noindent\textbf{Step 3:} Define the region

    \begin{align*}
      \Omega_j^t=&\{(y,t)\big||y_{n+1}|\leq W_j, d(y,l_t)\kappa\leq D_j\}\\
      \Omega_j=&\bigcup\limits_{t\in[-1-T_j,-1]} \Omega_j^t\\
      \partial^1\Omega_j=&\{(y,s)\in\partial\Omega_j|d(y,l_s)\kappa= D_j\}\\
      \partial^2\Omega_j=&\{(y,s)\in\partial\Omega_j\big||y_n|=W_j|\}
    \end{align*}
    where $D_j=2^{\frac{j}{100}}\Lambda_1$, $W_j=T_j=2^{\frac{j}{50}}\Lambda_1^2$.

    By repeatedly applying Lemma \ref{Vector Field closeness on Cylinder} and
    \ref{VF close Bowl times R}, we obtain a normalized set of rotation vector fields $\mathcal{K}^{(j)} =\{K_{\alpha}^{(j)}, 1\leq\alpha\leq \frac{(n-1)(n-2)}{2}\}$ such that

    \begin{align}
      &\max_{\alpha}|\left<K^j_{\alpha},\nu\right>|H\leq C(W_j+D_j+T_j)^2 2^{-j}\epsilon \text{ on }\partial^1\Omega_j^t \label{eq1}\\
      &\max_{\alpha}|\left<K^j_{\alpha},\nu\right>|H \leq C(W_j+D_j+T_j)^2 \epsilon \text{ on }\Omega_j \label{eq2}\\
      &\max_{\alpha}|K^j_{\alpha}|H \leq 2n\text{ on }\Omega_j \label{eq3}
    \end{align}

    (\ref{eq3}) follows from approximating by the normalized set of vector fields
     $\mathcal{K}^0 = \{K^0_{\alpha}(x) = J_{\alpha}x, 1\leq \alpha\leq \frac{(n-1)(n-2)}{2}\}$ and the fact that $\max_{\alpha}|K^0_{\alpha}|H\leq n$
     on Bowl$\times\mathbb{R}$ (see Appendix \ref{ode of bowl}).
      \\

     For each $K^j_{\alpha}$ ($1\leq \alpha\leq\frac{(n-1)(n-2)}{2}$, $j\geq 1$), let $u=\left<K_{\alpha}^j,\nu\right>$ on $M_t$
     and define the function
      $$f(x,t)=e^{-\Phi(x)+\lambda(t-\bar{t})}\frac{u}{H-\mu}$$
     where $\lambda,c$ will be determined later.

     $H$ and $u$ satisfy the Jacobi equation
     \begin{align*}
       \partial_t u & =\Delta u+|A|^2u\\
       \partial_t H & =\Delta H+ |A|^2H
     \end{align*}

     We get the evolution equation for $\frac{u}{H-\mu}$:
     \begin{align*}
       (\partial_t -\Delta )\left(\frac{u}{H-\mu}\right)=&\frac{(\partial_t -\Delta )u}{H-\mu}-
       \frac{u(\partial_t -\Delta )H}{(H-\mu)^2}+\frac{2\left<\nabla u,\nabla H\right>}{(H-\mu)^2}-\frac{2|\nabla H|^2u}{(H-\mu)^3}\\
       =&-\frac{cu|A|^2}{(H-\mu)^2}+2\left<\frac{\nabla H}{H-\mu},\nabla\left(\frac{u}{H-\mu}\right)\right>
     \end{align*}

     Then the evolution equation for $f$:
     \begin{align}\label{EQNforf}
       (\partial_t -\Delta )f=&e^{-\Phi+\lambda(t-\bar{t})}(\partial_t -\Delta )\left(\frac{u}{H-\mu}\right)+(\lambda-\partial_t\Phi+\Delta \Phi-|\nabla \Phi|^2)f  \notag\\
       &+2e^{-\Phi+\lambda(t-\bar{t})}\left<\nabla \Phi,\nabla\left(\frac{u}{H-\mu}\right)\right>\notag\\
       =&\left(\lambda-\frac{\mu|A|^2}{H-\mu}-\partial_t\Phi+\Delta \Phi-|\nabla\Phi|^2\right)f\notag\\
       &+2e^{-\Phi+\lambda(t-\bar{t})}\left<\nabla\left(\frac{u}{H-\mu}\right),\nabla \Phi+\frac{\nabla H}{H-\mu}\right>\notag\\
       =&\left(\lambda-\frac{\mu|A|^2}{H-\mu}-\partial_t\Phi+\Delta \Phi-|\nabla\Phi|^2\right)f+2\left<\nabla f,\nabla\Phi+\frac{\nabla H}{H-\mu}\right>\notag\\
       &+2\left<\nabla \Phi,\nabla\Phi+\frac{\nabla H}{H-c}\right>f\notag\\
       =&\left(\lambda-\frac{\mu|A|^2}{H-\mu}-\partial_t\Phi+\Delta \Phi +|\nabla\Phi|^2+2\frac{\left<\nabla \Phi,\nabla H\right>}{H-\mu}\right)f\\
       &+2\left<\nabla f,\nabla\Phi+\frac{\nabla H}{H-\mu}\right> \notag
     \end{align}

      Now let $\Phi(x)=\phi(x_{n+1})=\phi(\left<x,\omega_{n+1}\right>)$ where $\phi$ is a one variable function.

     We have the following computations:
     \begin{itemize}
       \item $\partial_t\Phi=\phi'(x_{n+1})\left<\partial_t x, \omega_{n+1}\right>=\phi'(x_{n+1})\left<\vec{H},\omega_{n+1}\right>$
       \item $\Delta\Phi=\phi'(x_{n+1})\Delta x_{n+1}+\phi''(x_{n+1})|\nabla x_{n+1}|^2=\phi'(x_{n+1})\left<\vec{H},\omega_{n+1}\right>+\phi''(x_{n+1})|\omega_{n+1}^T|^2$
       \item $\left<\nabla\Phi,\nabla H\right>=\phi'(x_{n+1})\left<\nabla x_{n+1},\nabla H\right>=\phi'(x_{n+1})\left<\omega_{n+1}^{T},\nabla H\right>$
       \item $|\nabla \Phi|^2=\phi'(x_{n+1})^2|\nabla x_{n+1}|^2=\phi'(x_{n+1})^2|\omega_{n+1}^T|^2$
     \end{itemize}

      Here $x_{n+1}$ is a short hand of $\left<x,\omega_{n+1}\right>$, $\omega_{n+1}^T$ denotes the projection onto the tangent plane of $M_t$.
      In the second line we used the identity $\Delta x=\vec{H}$.

      Since $\left<\vec{H},\omega_{n+1}\right>=\left<\omega_{n+1},\nabla H\right>=0$ on $\kappa^{-1}\Sigma\times\mathbb{R}$ and the curvature is
       bounded by $C\kappa$, by approximation we have:
     \begin{align*}
       |\omega_{n+1}^T-\omega_{n+1}| + |\left<\nabla H,\omega_{n+1}\right>| + |\left<\vec{H},\omega_{n+1}\right>|\leq C(L_1)\epsilon_1
     \end{align*} in
     $\hat{\mathcal{P}}(\bar{x},0,L_1,L_1^2)$.

     Therefore we have the estimate:
     \begin{align}\label{Coeff Est D Phi}
       |\partial_t\Phi| & \leq C(L_1)\epsilon_1|\phi'(x_{n+1})|\notag\\
       |\Delta\Phi|  & \leq C(L_1)\epsilon_1|\phi'(x_{n+1})|+|\phi''(x_{n+1})| \notag\\
       |\left<\nabla\Phi,\nabla H\right>| & \leq  C(L_1)\epsilon_1|\phi'(x_{n+1})| \notag\\
       |\nabla \Phi|^2&\leq\phi'(x_{n+1})^2
     \end{align}
     in $\hat{\mathcal{P}}(\bar{x},0,L_1,L_1^2)$.

     By the asymptotic of the Bowl soliton and approximation, there is a constant $c_n \in(0,1)$ depending only on $n$ such that
     $H(x) \geq 2c_nd(x,l_t)^{-\frac{1}{2}}$ in $\hat{\mathcal{P}}(\bar{x},0,L_1,L_1^2)$ when $d(x,l_t)$ is large. Moreover $|A|^2\geq \frac{H^2}{n}$ in $\hat{\mathcal{P}}(\bar{x},0,L_1,L_1^2)$.

     Next we choose an even function $\phi\in C^2(\mathbb{R})$ satisfying the following:
     \begin{align}
       |\phi'| & \leq \frac{c_n}{n}D_j^{-1/2} \label{Choice of varphi1}\\
       |\phi''| & \leq \frac{c_n^2}{n}D_j^{-1} \label{Choice of varphi2}\\
       \phi(W_j)& \geq \log\left((W_j+D_j+T_j)^{20}\right) \label{Choice of varphi3}\\
       \phi(200n^{5/2}) & \leq \log\left(W_j+D_j+T_j\right) \label{Choice of varphi4}\\
       \phi(0) &= 0, \ \phi'>0 \text{ when } x>0 \label{Choice of varphi5}
     \end{align}

     We can take $$\phi(s)=\frac{c_n^2}{n}D_j^{-1}\log(\cosh(s))$$
     for large $j$. Note that $\log(\cosh(s))\in (s-1,s]$.

     It's straightforward to check (\ref{Choice of varphi1}) (\ref{Choice of varphi2}) (\ref{Choice of varphi5}).
     To check (\ref{Choice of        varphi3}) (\ref{Choice of        varphi4}), note that there is a $j_1$ depending only on $n, \Lambda_1$ such that if $j\geq j_1$, then
     \begin{align*}
       \phi(W_j) &\geq \frac{c_n^2}{n}D_j^{-1}(W_j-1)\\
       &\geq \frac{c_n^2}{2n}2^{\frac{j}{100}} \\
       & \geq \log((W_j+D_j+T_j)^{20})
     \end{align*}
     and
     \begin{align*}
       \phi(200n^{5/2}) &\leq \frac{c_n^2}{n}D_j^{-1}\cdot 200n^{5/2} \\
       &\leq \log(2^{\frac{j}{100}}) \\
       &\leq  \log\left(W_j+D_j+T_j\right)
     \end{align*}

     Now we let  $\lambda=\frac{c_n^2}{n}D_j^{-1}, \mu=c_n D_j^{-1/2}$, therefore $H\geq 2\mu$ in $\Omega_j$ and
     \begin{align}\label{Coeff A2 / H}
       \frac{\mu|A|^2}{H-\mu} \geq \frac{\mu H^2}{n(H-\mu)}\geq \frac{2\mu H}{n}\geq \frac{4\mu^2}{n}
     \end{align}

     With (\ref{Coeff Est D Phi}) (\ref{Coeff A2 / H}) (\ref{Choice of varphi1})-(\ref{Choice of varphi5}), we have:
     \begin{align*}
       &\lambda-\frac{\mu|A|^2}{H-\mu}-\partial_t\Phi+\Delta \Phi+|\nabla\Phi|^2+2\frac{\left<\nabla \Phi,\nabla H\right>}{H-\mu}\\
       \leq&\lambda-\frac{4\mu^2}{n}+C(L_1)\epsilon_1|\phi'(x_n)|+|\phi''(x_n)|+\phi'(x_n)^2+C(L_1)\epsilon_1|\phi'(x_n)|\mu^{-1} \\
       <& \frac{c_n^2}{n}D_j^{-1}-\frac{4c_n^2}{n}D_j^{-1}+\frac{c_n^2}{n}D_j^{-1}
       +(\frac{c_n}{n}D_j^{-\frac{1}{2}})^2 + C(L_1)\epsilon_1 < 0
     \end{align*}

     Then maximal principle applies to (\ref{EQNforf}), we have
     $\sup\limits_{\Omega_j}|f|\leq \sup\limits_{\partial\Omega_j}|f|$.

     The boundary $\partial\Omega_j=\partial^1\Omega_j\cup\partial^2\Omega_j\cup\Omega_j^{-1-T_j}$,
     we can estimate $f$ on each of them:
     \begin{itemize}
       \item  on the boundary portion $\partial^1\Omega_j$:
       \begin{align*}
                \sup\limits_{\partial^1\Omega_j}|f|\leq & \sup\limits_{\partial^1\Omega_j}\frac{|u|H}{(H-\mu)H}\\
                \leq &C(W_j+D_j+T_j)^2 2^{-j}\epsilon\cdot D_j\\
                \leq &2^{-\frac{j}{2}}C\epsilon
             \end{align*}
       \item on the portion $\partial^2\Omega_j$:
       \begin{align*}
         \sup\limits_{\partial^2\Omega_j}|f|\leq & \sup\limits_{\partial^2\Omega_j}e^{-\phi(W_j)}\frac{|u|H}{H(H-\mu)}\\
         <&e^{-\phi(W_j)}\cdot C(W_j+D_j+T_j)^2 \epsilon\cdot D_j\\
        \leq& C(W_j+D_j+T_j)^{-10}\epsilon \\
        \leq & 2^{-\frac{j}{5}} C\epsilon
       \end{align*}
       we used $\phi(W_j)\geq \log\left((W_j+D_j+T_j)^{20}\right)$
       \item on the portion $\partial^3\Omega_j$, we have
       \begin{align*}
         \sup\limits_{\Omega_j^{-1-T_j}}|f|\leq & e^{-\lambda T_j}\sup\limits_{\Omega_j}\frac{|u|H}{H(H-\mu)}\\
         \leq &e^{-\frac{c_n^2}{n}T_j D_j^{-1}}\cdot C(W_j+D_j+T_j)^2\cdot 32D_j\epsilon\\
         \leq &2^{-j} C\epsilon
       \end{align*}
       where the last inequality used $T_jD_j^{-1} = D_j\geq jn c_n^{-2}-C$.
     \end{itemize}

     Putting them together we obtain that:
     \begin{align}\label{maxf}
        \sup\limits_{\Omega_j}|f|\leq &\sup\limits_{\partial\Omega_j}|f|\leq  2^{-\frac{j}{5}}C\epsilon
     \end{align}

     By approximation, in the parabolic neighborhood
     $\hat{\mathcal{P}}(\bar{x},0,100n^{5/2},100^2n^5)$ we have
     \begin{itemize}
       \item $|x_{n+1}|<200n^{5/2}$
       \item $-2\cdot 10^4n^5\leq t\leq 0$
       \item $H\leq C$
     \end{itemize}
     We also assume that $j$ is appropriate such that
     \begin{align}\label{Condition on j 1}
       \hat{\mathcal{P}}(\bar{x},0, 100n^{5/2}, 100^2n^5)\subset\Omega_j
     \end{align}
     \begin{align}\label{Cond on j 2}
       \Omega_j\subset\text{Int}(j)
     \end{align}
     Hence,
     \begin{align}\label{finalineqn}
       |u(y,s)|H(y,s)= & e^{\phi(x_{n+1})-\lambda t}(H(y,s)-\mu)f(y,s)H(y,s)\notag\\
        \leq& e^{\phi(200n^{5/2})+2\cdot10^4n^5\lambda}\cdot 2^{-\frac{j}{5}}C\epsilon\notag \\
       \leq&(W_j+D_j+T_j)\cdot 2^{-\frac{j}{5}}C\epsilon \notag\\
       \leq& 2^{-\frac{j}{10}} C_1\epsilon
     \end{align}
     in $\hat{\mathcal{P}}(\bar{x},0,100n^{5/2},100^2n^5)$. Here $C_1$
     depends only on $L_0,\epsilon_0,n$.

     We pick constants in the following order:
     First we can find $j_2\geq j_1$ depending only on $L_0,\epsilon_0,n$ such that
     (\ref{Condition on j 1}) (\ref{Cond on j 2}) holds with $j=j_2$ and
     \begin{align}\label{choice of j}
       2^{-\frac{j_2}{10}}C_1\epsilon <\frac{\epsilon}{2}
     \end{align}
     Next we pick $L_1$ large enough, finally we take $\epsilon_1$ small enough.
     (Therefore, $L_1$ may depend on $j_2, L_0, \epsilon_0, n$ and $\epsilon_1$ may depend on $L_1, j_2, L_0, \epsilon_0, n$)

     With such choice of constants, we can take $j=j_2$. Since
     (\ref{eq3}) ,(\ref{finalineqn}) 
     applies to every
     $\alpha\in \{1,2,...\frac{(n-1)(n-2)}{2}\}$, we obtain that
     \begin{itemize}
       \item $\max_{\alpha}|\left<\bar{K_{\alpha}},\nu\right>|H<\frac{\epsilon}{2}$
       \item $\max_{\alpha}|\bar{K_{\alpha}}|H<5n$
     \end{itemize}
     in $\hat{\mathcal{P}}(\bar{x},0,100n^{5/2},100^2n^5)$.

     By definition $(\bar{x},\bar{t}) = (\bar{x},0)$ is $\frac{\epsilon}{2}$ symmetric.

\end{proof}

 \section{Canonical neighborhood Lemmas and the proof of the main theorem}\label{section proof of main theorem}
 In this section we prove the main theorem. While this section is mostly the same as the Section 4 of \cite{zhu2020so}, we decide to contain most of the argument here for readers' convenience.

    Recall that the translating soliton satisfies the equation $ H = \left<V,\nu\right>$ for some fixed nonzero vector $V$, where $\nu$ is the inward pointing normal vector. With a translation and dilation we may assume that $V=\omega_{n}$ is a unit vector. The equation then becomes:
\begin{align}\label{translatoreqn1}
  H &= \left<\omega_n,\nu\right>
\end{align}
  $M_t = M+t\omega_n$ is an mean curvature flow.
    Let the height function in the space time to be
\begin{align}\label{height def}
  h(x,t)=\left<x,\omega_n\right>-t
\end{align}

  Throughout this section the mean curvature flow solution is always assumed to be embedded and complete.

\begin{Def}
  A mean convex mean curvature flow solution $M_t^n\subset\mathbb{R}^{n+1}$
  is said to be uniformly $k$-convex, if there is a positive constant $\beta$
  such that
  \begin{align*}
    \lambda_1+...+\lambda_k\geq \beta H
  \end{align*}
  along the flow, where $\lambda_1\leq\lambda_2\leq ...\leq \lambda_n$ are principal curvatures.
\end{Def}

\begin{Def}
  For any ancient solution $M_t$ defined on $t<0$, a blow down limit, or a tangent flow at $-\infty$, is the limit flow of $M^j_t=c_j^{-1}M_{c_j^2t}$ for some sequence $c_j\rightarrow\infty$, if the limit exists.
\end{Def}

    If $M$ is mean convex and noncollapsed ancient solution, then at least one blow down limit exists and any blow down sequence $c_j^{-1}M_{c_j^2t}$ has a subsequence that converges smoothly to $S^k_{\sqrt{-2kt}}\times\mathbb{R}^{n-k}$ for some $k=0,1,...,n$ with possibly a rotation, see e.g. \cite{haslhofer2017mean}, \cite{sheng2009singularity}, \cite{white2003nature}, \cite{white2000size}, \cite{huisken1999convexity}.

    Recall that the Gaussian density of a surface is defined by:
    \begin{align*}
      \Theta_{x_0,t_0}(M)=\int_{M}\frac{1}{(4\pi t_0)^{\frac{n}{2}}}e^{-\frac{|x-x_0|^2}{4t_0}}d\mu
    \end{align*}

    Using Huisken's monotoncity formula \cite{huisken1990asymptotic}, we will have the following:

    \begin{Lemma}\label{entropy limit lemma}
        For any mean convex and noncollapsed ancient solution $M_t$ defined on $t<0$, suppose that $M^{\infty}_t$ is a blow down limit, then $M_t^{\infty}$ must be the same up to rotation.
    \end{Lemma}
    \begin{proof}
      Suppose that $M^{\infty}_t$ is the limiting flow of $M^j_t=c_j^{-1}M_{c_j^2t}$ for some sequence $c_j\rightarrow\infty$. By the previous discussion, $M^{\infty}_t$ must be one of the self-similar generalized cylinders $S^k_{\sqrt{-2kt}}\times\mathbb{R}^{n-k}$.
      Moreover, mean convex ancient solution must be convex by the convexity estimate \cite{huisken1999convexity} (also c.f \cite{haslhofer2017mean}). Therefore the convergence is smooth with multiplicity one.
      It suffices to only consider the time slice $M^{\infty}_{-1}$ by the scale invariance of the entropy.

      By Huisken's monotonicity formula (\cite{huisken1990asymptotic}),
      $\Theta_{x_0,t_0+s_0-t}(M_{t})$  is monotone increasing in $t$.
      Hence $$\Theta = \lim\limits_{t\rightarrow\infty} \Theta_{x_0,t_0+s_0-t}(M_{t})\leq \infty$$ exists.

      By the scaling property of $\Theta$ we can compute
      \begin{align}\label{entropy bound 1}
        \Theta_{c_j^{-1}x_0,c_j^{-2}(t_0+s_0)-t}(c_j^{-1}M_{c_j^2t})=\Theta_{x_0,t_0+s_0-c_j^2t}(M_{c_j^2t})
      \end{align}
      whenever $c_j^2t<s_0$.

      By convexity, $\text{Vol}(M_t\cap B_R(0))\leq CR^n$ for some uniform constant $C$, so $F_{x',t'}(M_t\backslash B_R(0))\leq Ce^{-R^2/8}$ for some uniform constant $C$, whenever $x'$ and $\log t'$ are bounded.

      Taking $t=-1$ and letting $j$ large in (\ref{entropy bound 1}).
      Since $c_j^{-1}x_0\rightarrow 0, c^{-2}_j(t_0+s_0)-t\rightarrow 1$ and $c_j^{-1}M_{-c_j^2}$ converge smoothly to $M^{\infty}_{-1}$, we have
      \begin{align*}
       \Theta_{0,1}(M^{\infty}_{-1})=\lim\limits_{j\rightarrow\infty}\Theta_{c_j^{-1}x_0,c_j^{-2}(t_0+s_0)-t}(c_j^{-1}M_{c_j^2t})=\Theta
      \end{align*}
      It's easy to compute that $\Theta_{0,1}(S^{k}\times\mathbb{R}^{n-k})$, $k=0,1,...,n$ are all different numbers, we know that $M_{-1}^{\infty}$ must have the same shape, which means they must be the same up to rotation.

      \begin{Rem}
        By the work of Colding and Minicozzi \cite{colding2015uniqueness}, the blow down limit is actually unique (without any rotation), but we don't need this strong result in this paper.
      \end{Rem}

    \end{proof}

    In particular if $M_t$ is a translating solution, we can interpret the blow down process in a single time slice:
    \begin{Cor}\label{blow down for translator}
    Given $M^n\subset\mathbb{R}^{n+1}$ a strictly mean convex, noncollapsed translator which satisfies (\ref{translatoreqn1}),
    then for any $R>1, \epsilon>0$, there exists a large $C_0$ such that, if $a\geq C_0$ then $a^{-1}(M-a^2\omega_n)\cap B_{R}(0))$ is $\epsilon$ close to  a $S^k_{\sqrt{2k}}\times\mathbb{R}^{n-k} $ with some rotation for a fixed $1\leq k\leq n - 1$ in $C^{10}$ norm.
    \end{Cor}

    In the following we prove some canonical neighborhood Lemmas.
    \begin{Lemma}\label{canonical nbhd lemma}
      Let $M^n\subset\mathbb{R}^{n+1}$  to be a  noncollapsed, mean convex, uniformly 3-convex smooth translating soliton of mean curvature flow.\  \ $p\in M$ is a fixed point. Set $M_t$ to be the associated translating solution of the mean curvature flow. Suppose that one blow down limit of $M_t$ is $S^{n-2}\times\mathbb{R}^2$ and that the sub-level set $M_t\cap\{h(\cdot,t)\leq h_0\}$ is compact for any $h_0\in\mathbb{R}$.
      Then for any given  $L>10, \epsilon>0$,  there exist a constant $\Lambda$ with the following property:
      If $x\in M$ satisfies $|x-p|\geq\Lambda$, then after a parabolic rescaling by the factor  $H(x,t)$,
     the parabolic neighborhood $\hat{\mathcal{P}}(x,t,L,L^2)$ is $\epsilon$ close to the corresponding piece of the shrinking $S^{n-2}\times \mathbb{R}^2$, or the translating $\text{Bowl}^{n-1}\times \mathbb{R}$.

    \end{Lemma}

    \begin{proof}
      If $M$ is not strictly mean convex, then by strong maximal principle $M$ must be flat plane, therefore all the blow down limit is flat plane, this is a contradiction.

      Without loss of generality we assume that $p$ is the origin, and that $\omega_n$ is the unit vector in $x_n$ axis which points to the positive part of $x_n$ axis.

    Argue by contradiction. Suppose that the conclusion is not true, then there exist a sequence of points $x_j\in M$ satisfying $|x_j-p|\geq j$ but $\hat{\mathcal{P}}(x_j,0,L,L^2)$ is not $\epsilon$ close to either one of the models after appropriate parabolic rescaling.

  By the long range curvature estimate (c.f. \cite{white2000size}, \cite{white2003nature}, \cite{haslhofer2017mean}),  $H(p)|x_j-p|\rightarrow \infty$ implies
  \begin{align}\label{long range curvature estimate}
    H(x_j)|x_j-p|\rightarrow \infty
  \end{align}

  After passing to a subsequence, we may assume
  \begin{align}\label{distancedoubling}
    |p-x_{j+1}|\geq 2|p-x_j|
  \end{align}

  Now let $M^{(j)}_t=H(x_j)\Big(M_{H^{-2}(x_j)t}-x_j\Big)$ be a sequence of rescaled solutions. Under this rescaling $M^{(j)}_t$ remains eternal, $x_j$ is sent to the origin and $H_{M_0^{(j)}}(0)=1$.
  By the global convergence theorem (c.f \cite{haslhofer2017mean})  $M^{(j)}_t$ converges to an ancient solution $M^{\infty}_t$ which is  smooth,  non-collapsed, weakly convex, uniformly 3-convex and $H_{M^{\infty}_0}(0)=1$. Thus $M^{\infty}_0$ is strictly mean convex.
  Denote by $K^{\infty}$ the convex domain bounded by $M^{\infty}$.

  Suppose that $p$ is sent to $p^{(j)}$ in the $j^{th}$ rescaling. Then  $p^{(j)}=-H(x_j)x_j$ and $|p^{(j)}|\rightarrow \infty$ by (\ref{long range curvature estimate}).
  After passing to a subsequence, we may assume that $\frac{p^{(j)}}{|p^{(j)}|}=-\frac{x_j}{|x_j|}\rightarrow \Theta \in S^{n}$.
  Suppose that $l$ is the line in the direction of $\Theta$, i.e. $l=\{s\Theta\ | \ s\in\mathbb{R}\}$.

  Since rescaling doesn't change angle, we have $\angle x_jpx_{j+1}\rightarrow 0$. \\

  By elementary triangle geometry and (\ref{distancedoubling}) we have
  \begin{align}\label{edgerelation1}
    \frac{1}{2}|x_{j+1}-p|<|x_{j+1}-p|-|x_j-p|<|x_j-x_{j+1}| < |x_{j+1}-p|
  \end{align}
  Consequently $\angle x_{j+1}x_jp)\leq \angle x_jpx_{j+1})\rightarrow 0$. Therefore $\angle x_{j+1}x_jp\rightarrow \pi$.

  This implies that the limit $M_0$ contains the  line $l$. By convexity the first principal curvature vanishes. Then strong maximal principle (c.f \cite{hamilton1986four}, \cite{haslhofer2017mean},  \cite{white2003nature})
  implies that $M_t^{\infty}$ split off a line. That is,
  $M ^{\infty}_t=\mathbb{R}\times M'_t$ where $\mathbb{R}$ is in the direction of $l$ and $M'_t$ is an ancient solution that is non-collapsed, strictly mean convex, uniformly 2-convex. \\

  \noindent\textit{Case 1: } $M'_t$ is noncompact

  By the classification result of Brendle and Choi \cite{brendle2018uniqueness} \cite{brendle2019uniqueness}, we conclude that, up to translation and rotation,  $M'_t$ is  the shrinking $S^1\times\mathbb{R}$, or the translating $\text{Bowl}^{2}$.
    This means that for large $j$, the parabolic neighborhood $\hat{\mathcal{P}}^{(j)}(0,0,L,L^2)$ is $\epsilon$ close to either $S^{n-2}\times\mathbb{R}^2$ or $\text{Bowl}^2\times\mathbb{R}$, a contradiction.\\

    \noindent\textit{Case 2: }  $M'_t$ is compact

     \noindent\textit{Claim: } $l$ is the $x_n$ axis, consequently $\frac{x_j}{|x_j|}$ converges to $\omega_n$.

    In fact, let's denote  by $l^{\perp}$ the plane perpendicular to $l$ that passes through the origin. (Note that $l$ also passed through the origin by definition).
    Hence $l^{\perp}\cap M_t^{\infty}$ is isometric to $M_t'$.
    Suppose that $l$ is not the $x_n$ axis, then there is a unit vector $v\perp l$ such that $\left<v,\omega_n\right><0$.
    Since $M_t^{\infty}$ splits off in the direction $l$ with cross section a closed surface $M_t'$, we can find a point $y_j'\in M_0^{\infty}\cap l^{\perp}$ such that the inward normal vector $\nu$ of $M_0^{\infty}$ at $y_j'$ is equal to $v$.
    Then for large $j$ there is a point $y_j\in M_0^{(j)}\cap l^{\perp}$ such that the unit inward normal $\nu_{M^{(j)}_0}(y_j)$ is sufficiently close to $v$.  But this implies that $H_{M^{(j)}_0}(y_j)=H_{M_0}(x_j)^{-1}\left<\nu_{M^{(j)}_0}(y_j),\omega_n\right><0$ a contradiction.

    By the above argument we know that the $\mathbb{R}$ direction is parallel to $\omega_n$.
    Hence the cross section perpendicular to the $\mathbb{R}$ factor are the level set of the height function.
    Moreover the cross section $M_0^{(j)}\cap \{h(\cdot,0)=0\}$  converge smoothly to the cross section $(\mathbb{R}\times M_0')\cap \{h(\cdot,0)=0\}$, which is $M_0'$.

     Now let's convert it to the unrescaled picture. Let $h_j=h(x_j,0)$ and $N_j=M\cap \{h(\cdot,0)=h_j\}$ ($N_j$ should be considered as  a hypersurface in $\mathbb{R}^n$). Then $N_j$, after appropriate rescaling, converges to $M_0'$.

    Define a scale invariant quantity for hypersurface(possibly with boundary):
    \begin{align*}
      \text{ecc}(N_j)=\text{diam}(N_j)\sup H_{N_j}
    \end{align*}
     where diam denotes the extrinsic diameter and $H_{N_j}$ is the mean curvature of $N_j$ in $\mathbb{R}^n$.
    Since $N_j\rightarrow M_0'$, we have:
    \begin{align}\label{ecc M0'}
      \text{ecc}(N_j)< 2\text{ecc}(M_0')
    \end{align}
    for all large $j$.

    On the other hand, by assumption each sub-level set $\{h(\cdot, 0)\leq h_0\}$ is compact, this means $h_j\rightarrow+\infty$.
    By Corollary \ref{blow down for translator}, for any $\eta>0, R\gg \sqrt{2n}$ there exists $A_j\in SO(4)$ for each sufficiently large $j$ such that
    $\sqrt{h_j}^{-1}(M_0-h_j\omega_n)\cap B_R(0)$ is $\eta$ close to $\Sigma_j:=A_j(S^{n-2}_{\sqrt{2(n-2)}}\times\mathbb{R})$, which is a rotation of $S^{n-2}_{\sqrt{2(n-2)}}\times\mathbb{R}$.
    Since $\left<\nu,\omega_3\right>>0$ on $\sqrt{h_j}^{-1}(M_0-h_j\omega_3)\cap B_R(0)$, by approximation $\left<\nu,\omega_3\right>>-C\eta$ on $\Sigma_j\cap B_R(0)$.

    Since $\nu$ is arbitrary and $R$ is large, we have $|\left<\nu,\omega_3\right>|<C\eta$ along the $S^{n-2}$ fiber of $\Sigma$.

    This means that the $\mathbb{R}^2$ factor of $\Sigma_j$ must be almost perpendicular to the level set of the height function $h$.
    Consequently, the intersection of $\Sigma_j\cap B_R(0)$ with $\{h(\cdot,0)=0\}$ must be $C\eta$ close to some rotated $(S^{n-2}_{\sqrt{2(n-2)}}\times\mathbb{R})\cap B_R(0)$,
    hence the mean curvature (computed in $\mathbb{R}^n$) is at least $\sqrt{\frac{n-2}{2}}-C\eta$.

    Now let $\tilde{N}_j=\sqrt{h_j}^{-1}(M_0-h_j\omega_3)\cap\{h(\cdot,0)=0\}\cap B_R(0)$. Then $\tilde{N}_j$, considered as surface in $\mathbb{R}^n$, is $C\eta$ close to $\Sigma_j\cap\{h(\cdot,0)=0\}\cap B_R(0)$. Therefore
    \begin{align}\label{ecc cylinder}
      \text{ecc}(\tilde{N}_j)\geq (\sqrt{\frac{n-2}{2}}-C\eta)R
    \end{align}

    Now we take $\eta<\frac{1}{8C}$, $R>4\text{ecc}(M_0')$ and $j$ large accordingly. Note that $\tilde{N}_j=\sqrt{h_j}^{-1}(N_j-h_j\omega_n)$, then the scale invariance of ecc together with (\ref{ecc M0'}) and (\ref{ecc cylinder})  lead to a contradiction.

    \end{proof}

    \begin{Lemma}\label{blowdownnecklemma}
      Suppose that $M^n\subset\mathbb{R}^{n+1}$ is  noncollapsed, convex, uniformly 3-convex translating soliton of mean curvature flow.  Set $M_t$ to be the associated translating solution. Suppose that one blow down limit of $M_t$ is $S^{n-1}\times\mathbb{R}$  and that the sub-level set $M_t\cap\{h(\cdot,t)\leq h_0\}$ is compact for any $h_0\in\mathbb{R}$. Then $M=\text{Bowl}^n$.
    \end{Lemma}
    \begin{proof}
    We use $M_t$ to denote the associated translating solution of the mean curvature flow.
    If $M$ is not strictly convex, then it split off a line. Since one of the blow down limit is $S^{n-1}\times\mathbb{R}$, it must be shrinking cylinders, a contradiction.
    Fixing  $p\in M$.
    Assume without loss of generality  that $p$ is the origin and  $\omega_n$ is the unit vector in $x_n$ axis which points to the positive part of $x_n$ axis.

    Using the the argument from the Case 2 in Lemma \ref{canonical nbhd lemma}, the ball $B_{100\sqrt{a}}(a\omega_3)\cap M$ is close a $S^{n-1}_{\sqrt{2(n-1)}}\times\mathbb{R} $ with some translation and rotation for all large $a$.
    Moreover, the fact that $\left<\nu, \omega_n\right> >0$ and approximation implies that the $\mathbb{R}$  direction is almost parallel to the $x_n$ axis, and the cross section $\{h(\cdot,0)=a\}$ is close to  $S^{n-1}$ after some scaling for sufficiently large $a$.
    In particular $M\cap \{h(\cdot,0)\geq A\}$ is uniformly 2-convex for a fixed large $A$.
   Since the sub-level set $\{h(\cdot, 0)\leq A\}$ is compact and $M$ is strictly convex, we know that $M$ is uniformly 2-convex.

    Therefore $M$ is noncompact and strictly convex, the classification result (c.f \cite{haslhofer2015uniqueness}, \cite{brendle2018uniqueness})applies and $M=\text{Bowl}^n$.

    \end{proof}

\begin{proof}[Proof of Theorem \ref{Main Theorem}]

      Let $M_t=M+t\omega_n$ be the associated mean curvature flow. Denote by $g$ the metric on $M$ and $g_t$ the metric on $M_t$.

      By the convexity estimate (c.f \cite{haslhofer2017mean}),  $M$ must be convex.
      If $M$ is not strictly convex, then by maximal principle \cite{hamilton1986four} it must split off a line, namely it is a product of a line and a convex, uniformly 2-convex, non-collapsed tranlating solution (which must be noncompact). By the classification result \cite{brendle2019uniqueness} \cite{brendle2018uniqueness} or \cite{haslhofer2015uniqueness}, $M=$Bowl$^{n-1}\times\mathbb{R}$  and we are done.

      We may then assume that $M$ (or $M_t$) is strictly convex. It suffices to find a normalized rotation vector field that is tangential to $M$.

      Suppose that $M_t$ bounds the open domain $K_t=K+t\omega_n$ and attains maximal mean curvature at the tip $p_t=p+t\omega_n$.
      After some rotation and  translation we assume that $p$ is the origin and $\omega_n$ is the unit vector in $x_n$ axis which points to the positive part of $x_n$ axis.

      As in  \cite{haslhofer2015uniqueness}, we can show that $\omega_n^T=0$ at tip $p$ : first take the gradient of (\ref{translatoreqn1}) at the point $p$:
\begin{align*}
  \nabla H &= A(\omega_n^{T})
\end{align*}
where $A$ is the shape operator, which is non-degenerate because of the strict convexity.
Since $\nabla H=A(\omega_n^T)=0$ at $p$, we have $\omega_n^T=0$, moreover $\vec{H}=\omega_n$. In particular $\nu(p)=\omega_n$ and $H(p)=1$, thus $p_t$ is indeed the trajectory that moves by mean curvature.

      By the convexity, $K$ is contained  in the upper half space $\{x_n\geq 0\}$. The strict convexity implies that there is a cone $\mathcal{C}_{\eta}=\{x\in\mathbb{R}^{n+1}| x_n\geq \eta\sqrt{x_1^2+...+x_{n-1}^2+x_{n+1}^2}\}$ ($0<\eta<1$) such that
      \begin{align}\label{m contained in cone}
        K_t\backslash B_{1}(p)\subset \mathcal{C}_{\eta}\backslash B_{1}(p)
      \end{align}

      The height function  $h(x,t)=\left<x,\omega_n\right>-t$ \ then measures the signed distance from the support plane $\{x_n=t\}$ at the tip.
      (\ref{m contained in cone}) implies that any sub-level set $\{h(\cdot, 0)\leq h_0 \}$ is compact.

       By Corollary \ref{blow down for translator}, a blow down limit of $M_t$ exists and is uniformly 3-convex, thus must be  $S^{k}_{\sqrt{-2kt}}\times\mathbb{R}^{n-k}$ up to rotation,  for a fixed $k=n-2,n-1,n$. If $k=n$, then $M$ is compact, thus can't be translator.
      If $k=n-1$, by Lemma \ref{blowdownnecklemma}, $M$ is Bowl$^n$, and the result follows immediately.
      So from now on we assume that  $k=n-2$ and Lemma \ref{canonical nbhd lemma} is applicable.

      By Lemma \ref{canonical nbhd lemma}, there exists $\Lambda'$ depending on $L_1,\epsilon_1$ given in the Theorem \ref{bowl x R improvement} such that if $(x,t)\in M_t$ satisfies $h(x,t)\geq \Lambda'$, then $\hat{\mathcal{P}}(x,t,L_1,L_1^2)$ is $\epsilon_1$ close to either a piece of translating Bowl$^{n-1}\times\mathbb{R}$ or a family of shrinking cylinder $S^{n-2}\times\mathbb{R}^2$ after rescaling.

      By the equation (\ref{translatoreqn1}) $\partial_t h(x,t)=\left<\nu,\omega_n\right>-1\leq 0$, so $h$ is nonincreasing.

      Using (\ref{m contained in cone}) we have $h(x,t)\geq \frac{\eta}{2}|x-p_t|$ whenever $|x-p_t|\geq 1$. Also $h(x,t)\leq|x-p_t|$.

     By the long range curvature estimate (c.f. \cite{white2000size} \cite{white2003nature}, \cite{haslhofer2017mean}) there exists  $\Lambda>\Lambda'$ such that
     $H(x,t)|x-p_t|>2\cdot10^3\eta^{-1}L_1$ for all $(x,t)$ satisfying $h(x,t)>\Lambda$. Consequently $H(x,t)h(x,t)>10^3L_1$.

     Define

     $\Omega_j=\{(x,t)\ |\ x\in M_t,\ t\in[-2^{\frac{j}{100}}, 0],\ h(x,t)\leq2^{\frac{j}{100}}\Lambda \}$

     $\partial^{1} \Omega_j = \{(x,t)\ |\ x\in M_t,\ t\in[-2^{\frac{j}{100}}, 0],\ h(x,t)=2^{\frac{j}{100}}\Lambda \}$

     $\partial^{2} \Omega_j = \{(x,t)\in\partial\Omega_j \ |\ t=-2^{\frac{j}{100}} \}$

      \noindent\textbf{Step 1:} If $x\in M_t$ satisfies $h(x,t)\geq 2^{\frac{j}{100}}\Lambda$, then $(x,t)$ is $2^{-j}\epsilon_1$ symmetric.

      When $j=0$ the statement is true by the choice of $\Lambda', \Lambda$. If the statement is true for $j-1$. Given $x\in M_t$ satisfying $h(x,t)\geq 2^{\frac{j}{100}}\Lambda$, the choice of $\Lambda$ ensures that
      $L_1H(x,t)^{-1}<10^{-3}h(x,t)$. So every point in the geodesic ball $B_{g_t}(x,L_1H(x)^{-1})$ has height at least $(1-10^{-3})h(x,t)\geq 2^{\frac{j-1}{100}}\Lambda$. Moreover, the height function is nonincreasing in time $t$ (equivalently nondecreasing backward in time $t$) so we conclude that  the parabolic neighborhood $\hat{\mathcal{P}}(x,t,L_1,L_1^2)$ is contained in the set $\{(x,t) | x\in M_t, h(x,t)\geq 2^{\frac{j-1}{100}}\Lambda\}$. In particular every point in
      $\hat{\mathcal{P}}(x,t,L_1,L_1^2)$ is $2^{-j+1}\epsilon_1$ symmetric by induction hypothesis.
      Now we can apply either Lemma \ref{Cylindrical improvement} or Lemma \ref{bowl x R improvement} to obtain that $(x,t)$ is $2^{-j}\epsilon$ symmetric.\\

      \noindent\textbf{Step 2:} The intrinsic diameter of the set $M_t\cap\{h(x,t)=a\}$ is bounded by $6\eta^{-1}a$ for $a\geq 1$.

      It suffices to consider $t=0$. Since $M$ is convex, the level set $\{h(x,t)=a\}$ is also convex. By (\ref{m contained in cone}), $M \cap\{h(x,0)=a\}$ is contained in a 3 dimensional ball of radius $\eta^{-1}a$, thus has extrinsic diameter at most $2\eta^{-1}a$. The intrinsic diameter of a convex set is bounded by triple of the extrinsic diameter, the assertion then follows immediately. \newline

      \noindent\textbf{Step 3:} For each $j$, by repeatedly applying Lemma \ref{Vector Field closeness on Cylinder} and Lemma \ref{VF close Bowl times R}, there eixsts a single normalized set of rotation vector fields $\mathcal{K}^{(j)} = \{K^{(j)}_{\alpha}, 1\leq\alpha\leq \frac{(n-2)(n-1)}{2}\}$ satisfying $\max_{\alpha}|\left<K^{(j)}_{\alpha},\nu\right>|H\leq 2^{-\frac{j}{2}}C\epsilon_1$ on $\partial^{1} \Omega_j$

      For each $\alpha$, Define the function $f^{(j)}$ on $\Omega_j$ to be
      \begin{align}\label{f in the final step}
        f^{(j)}(x,t)=\exp (2^{\lambda_j t})\frac{\left<K_{\alpha}^{(j)},\nu\right>}{H-c_j}
      \end{align}
      where $\lambda_j=2^{-\frac{j}{50}}, c_j=2^{-\frac{j}{100}}$.

      As in \cite{brendle2019uniqueness} or (\ref{EQNforf}), we have
      \begin{align*}
        (\partial_t -\Delta )f^{(j)}
       =&\left(\lambda_j-\frac{c_j|A|^2}{H-c_j} \right)f^{(j)}+2\left<\nabla f^{(j)},\frac{\nabla H}{H-c_j}\right>
      \end{align*}

      Since $H\geq 10^{3}h^{-1}$ whenever $h>\Lambda$, we have $H>n\cdot 2^{-\frac{j}{100}}=nc_j$ in $\Omega_j$ for all large $j$.

      Hence
      \begin{align*}
        \lambda_j-\frac{c_j|A|^2}{H-c_j}\leq \lambda_j-\frac{c_jH^2}{n(H-H/2)}
        \leq \lambda_j-2c_j^2<0
      \end{align*}

      By the maximal principle,
      \begin{align*}
        \sup\limits_{\Omega_j}|f^{(j)}|\leq & \sup\limits_{\partial\Omega_j}|f^{(j)}|=
        \max\left\{\sup\limits_{\partial^1\Omega_j}|f^{(j)}|,\sup\limits_{\partial^2\Omega_j}|f^{(j)}|\right\}
      \end{align*}

      Since
      \begin{align*}
        \sup\limits_{\partial^1\Omega_j}|f^{(j)}|\leq \frac{|\left<K^{(j)},\nu\right>H|}{(H-c_j)H}\leq C\frac{2^{-\frac{j}{2}}\epsilon_1}{2c_j^2}\leq 2^{-\frac{j}{4}}C\epsilon_1
      \end{align*}

      Meanwhile,  $|\left<K^{(j)},\nu\right>|\leq C2^{\frac{j}{50}}$ in $\Omega_j$. Therefore, for large $j$:
      \begin{align*}
        \sup\limits_{\partial^2\Omega_j}|f^{(j)}|\leq \exp(-2^{\frac{j}{50}})\frac{|\left<K^{(j)},\nu\right>|}{c_j}\leq 2^{-j}
      \end{align*}

      Putting them together we get $|f^{(j)}|\leq 2^{-j/4}$ in $\Omega_j$ . Since this is true for each $\alpha$, we know that the axis of $\mathcal{K}^{(j)}$ (i.e the 0 set of all $K_{\alpha}^{(j)}$) has a uniform bounded distance from $p$ or equivalently  $\max_{\alpha}|K_{\alpha}^{(j)}(0)|\leq C$.
      (If this is not the case, then passing to subsequence we may assume that $\max_{\alpha}|K_{\alpha}^{(j)}(0)|\rightarrow \infty$, then at least one of
      $\tilde{K}_{\alpha}^{(j)}=\frac{K_{\alpha}^{(j)}}{\max_{\alpha}|K_{\alpha}^{(j)}(0)|}$ converges locally to nonzero constant vector field $\tilde{K}_{\alpha}$ which is tangential to $M_0$ near $0$, hence $M_0$ is not strictly convex, a contradiction.)

      Passing to a subsequence and taking limit, there exists normalized set of rotation vector fields $\mathcal{K}$ which is tangential to $M$. This completes the proof.
\end{proof}

\newpage
    \section{Appendix }

    \subsection{ODE for Bowl soliton}\label{ode of bowl}
    We will set the origin to be the tip of the Bowl soliton and the $x_{n+1}$ to be the translating axis.
    Use the parametrization: $\varphi:\mathbb{R}^n\rightarrow\text{Bowl}^n$:
    $\varphi(x)=(x,h(x))$ where $h(x)=\varphi(|x|)$ is a one variable function satisfying the ODE:
    \begin{align*}
      \frac{\varphi''}{1+\varphi'^2}+\frac{(n-1)\varphi'}{r}=1
    \end{align*}
    with initial condition $\varphi(0)=\varphi'(0)=0$.

    Since $\varphi''\geq0$, we have the inequality:
    \begin{align*}
      \varphi''+\frac{(n-1)\varphi'}{r}\geq 1
    \end{align*}
    Using the integrating factor, when $r>0$:
    \begin{align*}
      (r^{n-1}\varphi')'=r^{n-1}\left(\varphi''+\frac{(n-1)\varphi'}{r}\right)\geq r^{n-1}
    \end{align*}
    Integrating from $0$ to $r$ we obtain $\varphi'\geq \frac{r}{n}$ and $\varphi \geq \frac{r^2}{2n}$

    The mean curvature at $(x,h(x))$ is given by
    \begin{align*}
      \frac{1}{\sqrt{1+\varphi'(|x|)^2}}\leq \frac{1}{\sqrt{1+\frac{|x|^2}{n^2}}}<\frac{n}{|x|}
    \end{align*}

    As an application, if $J_{\alpha}$ is the antisymmetric matrix in $so(n-1)\subset so(n+1)$ with Tr$(J_{\alpha}J_{\alpha}^T)$=1 (i.e the unit vector in $so(n-1)\subset so(n+1)$),  Then $|J|H<2n$ on the Bowl soliton.

    \subsection{Heat Kernel estimate}\label{Appendix HKEST}
    The purpose is to justify (\ref{hkest2}). Let's use a more general notation, that is, we replace $\frac{L_0}{4}$ by $L$. Denote
       \begin{align*}
         &D(x,y,k_1,k_2,\delta_1,\delta_2)=
       \left|(x_1,x_2)-(\delta_1 y_1, \delta_2 y_2)-(1-\delta_1,1-\delta_2)L+(4k_1,4k_2)L\right|
       \end{align*}

        Then
        \begin{align*}
          K_t&(x,y)=-\frac{1}{4\pi t}\sum_{\delta_i\in \{\pm 1\}, k_i\in \mathbb{Z}}(-1)^{-(\delta_1+\delta_2)/2}\cdot e^{-\frac{D^2}{4t}}
        \end{align*}

       We observe that, when $x\in\Omega_{L/25}$ and $y\in\partial\Omega_{L}$,
       $$|D|\geq \frac{|k_1|+|k_2|+1}{2}L$$.

       Then $|\partial_{\nu_{y}}e^{-D^2/4t}|\leq \frac{|D|}{2t}
       e^{-D^2/4t}\leq C\frac{(|k_1|+|k_2|+1)L}{t}\exp\left(-\frac{(|k_1|+|k_2|+1)^2L^2}{16t}\right)$.
       The last inequality holds because the function $\lambda e^{-\lambda}$ is decreasing for $\lambda>1$.
       Consequently,
       \begin{align*}
         |\partial_{\nu_y} K_t(x,y)|&\leq \sum_{n=0}^{\infty}\sum_{|k_1|+|k_2|=n, \delta_i\in\{\pm1\}}\frac{1}{4t}|\partial_{\nu_{y}}e^{-\frac{D^2}{4t}}|\\
         &\leq C\sum_{n=0}^{\infty}\sum_{|k_1|+|k_2|=n} \frac{(|k_1|+|k_2|+1)L}{t^2}\exp\left(-\frac{(|k_1|+|k_2|+1)^2L^2}{16t}\right)\\
          & \leq C\sum_{n=1}^{\infty}\frac{n^2L}{t^2}\exp\left(\frac{-n^2L^2}{50t}\right)\exp\left(\frac{-n^2L^2}{50t}\right)\\
         &\leq C\sum_{n=1}^{\infty}\frac{n^2L}{t^2}\frac{2(50t)^2}{(n^2L^2)^2}e^{\frac{-L^2}{50t}} \leq \frac{C}{L^3}e^{\frac{-L^2}{50t}}
       \end{align*}
       where the last inequality used the fact that $e^{-s}< \frac{2}{s^2}$ when $s>0$. Integrating along the boundary and replacing $t$ by $t-\tau$ we get \begin{align*}
         \int_{\partial{\Omega_{L}}}|\partial_{\nu_y}K_{t-\tau}(x,y)|dy\leq \frac{C}{L^2}e^{\frac{-L^2}{50t}}\leq \frac{CL^2}{(t-\tau)^2}e^{-\frac{L^2}{50(t-\tau)}}
       \end{align*}
       The last inequality is because $t-\tau<L^2$.
       Putting $L_0/4$ in place of $L$ and then we get (\ref{hkest2}).

       \subsection{Intrinsic and extrinsic diameter of a convex hypersurface}\label{intrisic diam extrinsic diam appendix}

       Give a compact, convex set $K\subset\mathbb{R}^n$, the boundary $M=\partial K$. The intrinsic diameter of $M$ is
       \begin{align*}
         d_1(M)=\sup\limits_{x,y\in M}\inf\limits_{\substack{\gamma(0)=x, \gamma(1)=y\\ \gamma \text{ continous }}} L(\gamma)
       \end{align*}
       The extrinsic diameter of $M$ is
       \begin{align*}
         d_2(M)=\sup\limits_{x,y\in M}|x-y|
       \end{align*}
       We show that $d_1(M)\leq 3d_2(M)$.

       Given a two dimensional plane $P$ whose intersection with $M$ contains at least two points. $P\cap K$ is convex.  Let's restriction our attention to $P$.
       Take $x,y\in P\cap M$ that attains $d_2(P\cap M)$.  Without loss of generality we may assume that $P=\mathbb{R}^2$, $x=(-1,0), y=(1,0)$. Thus $d_2(P\cap M)=2$.

       $P\cap K$ is contained in the rectangle
       $R=\{(x,y)\in\mathbb{R}^2\ | \ |x|\leq 1, |y|\leq 2\}$ by the choice of $x,y$.
       Let $R^{\pm}$ be the upper/lower half of this rectangle, respectively.
       By convexity we see that $M\cap P\cap R^{+}$ is a graph of a concave function on $[-1,1]$, thus the length is bounded by half of the perimeter of $P$, which is $6$. In the same way $M\cap P\cap R^{-}$ has length at most $6$.

       Now let $x',y'$ attains $d_1(M)$. Then we find some $P$ passing through $x',y'$. The above argument shows that there exists a curve $\gamma$ connecting $x',y'$ with $L(\gamma)\leq 3d_2(M\cap P)$.
       Then we have:
       \begin{align*}
         d_1(M)\leq L(\gamma)\leq 3 d_2(M\cap P)\leq 3d_2(M)
       \end{align*}

\newpage

\nocite{*}
\bibliography{BIBROTATIONSYMHIGHDIM}

\providecommand{\bysame}{\leavevmode\hbox to3em{\hrulefill}\thinspace}
\providecommand{\MR}{\relax\ifhmode\unskip\space\fi MR }
\providecommand{\MRhref}[2]{%
  \href{http://www.ams.org/mathscinet-getitem?mr=#1}{#2}
}
\providecommand{\href}[2]{#2}
\begin{thebibliography}{HIMW19}

\bibitem[ADS19]{angenent2019unique}
Sigurd Angenent, Panagiota Daskalopoulos, and Natasa Sesum, \emph{Unique
  asymptotics of ancient convex mean curvature flow solutions}, J. Differential
  Geom. \textbf{111} (2019), no.~3, 381--455. \MR{3934596}

\bibitem[ADS20]{angenent2020uniqueness}
\bysame, \emph{Uniqueness of two-convex closed ancient solutions to the mean
  curvature flow}, Ann. of Math. (2) \textbf{192} (2020), no.~2, 353--436.
  \MR{4151080}

\bibitem[And12]{andrews2012noncollapsing}
Ben Andrews, \emph{Noncollapsing in mean-convex mean curvature flow}, Geom.
  Topol. \textbf{16} (2012), no.~3, 1413--1418. \MR{2967056}

\bibitem[BC18]{brendle2018uniqueness}
Simon Brendle and Kyeongsu Choi, \emph{Uniqueness of convex ancient solutions
  to mean curvature flow in higher dimensions}, arXiv:1804.00018 (2018).

\bibitem[BC19]{brendle2019uniqueness}
\bysame, \emph{Uniqueness of convex ancient solutions to mean curvature flow in
  {$\Bbb R^3$}}, Invent. Math. \textbf{217} (2019), no.~1, 35--76. \MR{3958790}

\bibitem[BLT20]{bourni2020existence}
Theodora Bourni, Mat Langford, and Giuseppe Tinaglia, \emph{On the existence of
  translating solutions of mean curvature flow in slab regions}, Anal. PDE
  \textbf{13} (2020), no.~4, 1051--1072. \MR{4109899}

\bibitem[Bre13]{brendle2013rotational}
Simon Brendle, \emph{Rotational symmetry of self-similar solutions to the
  {R}icci flow}, Invent. Math. \textbf{194} (2013), no.~3, 731--764.
  \MR{3127066}

\bibitem[Bre15]{brendle2015sharp}
\bysame, \emph{A sharp bound for the inscribed radius under mean curvature
  flow}, Invent. Math. \textbf{202} (2015), no.~1, 217--237. \MR{3402798}

\bibitem[CM12]{colding2012generic}
Tobias~H. Colding and William~P. Minicozzi, II, \emph{Generic mean curvature
  flow {I}: generic singularities}, Ann. of Math. (2) \textbf{175} (2012),
  no.~2, 755--833. \MR{2993752}

\bibitem[CM15]{colding2015uniqueness}
Tobias~Holck Colding and William~P. Minicozzi, II, \emph{Uniqueness of blowups
  and \l ojasiewicz inequalities}, Ann. of Math. (2) \textbf{182} (2015),
  no.~1, 221--285. \MR{3374960}

\bibitem[DHS10]{daskalopoulos2008classification}
Panagiota Daskalopoulos, Richard Hamilton, and Natasa Sesum,
  \emph{Classification of compact ancient solutions to the curve shortening
  flow}, J. Differential Geom. \textbf{84} (2010), no.~3, 455--464.
  \MR{2669361}

\bibitem[Ham86]{hamilton1986four}
Richard~S. Hamilton, \emph{Four-manifolds with positive curvature operator}, J.
  Differential Geom. \textbf{24} (1986), no.~2, 153--179. \MR{862046}

\bibitem[Ham95a]{hamilton1993formations}
\bysame, \emph{The formation of singularities in the {R}icci flow}, Surveys in
  differential geometry, {V}ol. {II} ({C}ambridge, {MA}, 1993), Int. Press,
  Cambridge, MA, 1995, pp.~7--136. \MR{1375255}

\bibitem[Ham95b]{hamilton1995harnack}
\bysame, \emph{Harnack estimate for the mean curvature flow}, J. Differential
  Geom. \textbf{41} (1995), no.~1, 215--226. \MR{1316556}

\bibitem[Has15]{haslhofer2015uniqueness}
Robert Haslhofer, \emph{Uniqueness of the bowl soliton}, Geom. Topol.
  \textbf{19} (2015), no.~4, 2393--2406. \MR{3375531}

\bibitem[HH16]{haslhofer2013ancient}
Robert Haslhofer and Or~Hershkovits, \emph{Ancient solutions of the mean
  curvature flow}, Comm. Anal. Geom. \textbf{24} (2016), no.~3, 593--604.
  \MR{3521319}

\bibitem[HIMW19]{Hoffman2019}
D.~Hoffman, T.~Ilmanen, F.~Mart\'{\i}n, and B.~White, \emph{Graphical
  translators for mean curvature flow}, Calc. Var. Partial Differential
  Equations \textbf{58} (2019), no.~4, Paper No. 117, 29. \MR{3962912}

\bibitem[HK17a]{haslhofer2017mean}
Robert Haslhofer and Bruce Kleiner, \emph{Mean curvature flow of mean convex
  hypersurfaces}, Comm. Pure Appl. Math. \textbf{70} (2017), no.~3, 511--546.
  \MR{3602529}

\bibitem[HK17b]{haslhofer2017}
\bysame, \emph{Mean curvature flow with surgery}, Duke Math. J. \textbf{166}
  (2017), no.~9, 1591--1626.

\bibitem[HS99]{huisken1999convexity}
Gerhard Huisken and Carlo Sinestrari, \emph{Convexity estimates for mean
  curvature flow and singularities of mean convex surfaces}, Acta Math.
  \textbf{183} (1999), no.~1, 45--70. \MR{1719551}

\bibitem[HS09]{huisken2009mean}
\bysame, \emph{Mean curvature flow with surgeries of two-convex hypersurfaces},
  Invent. Math. \textbf{175} (2009), no.~1, 137--221. \MR{2461428}

\bibitem[Hui84]{huisken1984flow}
Gerhard Huisken, \emph{Flow by mean curvature of convex surfaces into spheres},
  J. Differential Geom. \textbf{20} (1984), no.~1, 237--266. \MR{772132}

\bibitem[Hui90]{huisken1990asymptotic}
\bysame, \emph{Asymptotic behavior for singularities of the mean curvature
  flow}, J. Differential Geom. \textbf{31} (1990), no.~1, 285--299.
  \MR{1030675}

\bibitem[Hui93]{huisken54local}
\bysame, \emph{Local and global behaviour of hypersurfaces moving by mean
  curvature}, Differential geometry: partial differential equations on
  manifolds ({L}os {A}ngeles, {CA}, 1990), Proc. Sympos. Pure Math., vol.~54,
  Amer. Math. Soc., Providence, RI, 1993, pp.~175--191. \MR{1216584}

\bibitem[KMl14]{kleene2014self}
Stephen Kleene and Niels~Martin M\o~ller, \emph{Self-shrinkers with a
  rotational symmetry}, Trans. Amer. Math. Soc. \textbf{366} (2014), no.~8,
  3943--3963. \MR{3206448}

\bibitem[Per02]{perelman2002entropy}
Grisha Perelman, \emph{The entropy formula for the ricci flow and its geometric
  applications}, arXiv:0211159 (2002).

\bibitem[SW09]{sheng2009singularity}
Weimin Sheng and Xu-Jia Wang, \emph{Singularity profile in the mean curvature
  flow}, Methods Appl. Anal. \textbf{16} (2009), no.~2, 139--155. \MR{2563745}

\bibitem[Wan11]{wang2011convex}
Xu-Jia Wang, \emph{Convex solutions to the mean curvature flow}, Ann. of Math.
  (2) \textbf{173} (2011), no.~3, 1185--1239. \MR{2800714}

\bibitem[Whi00]{white2000size}
Brian White, \emph{The size of the singular set in mean curvature flow of
  mean-convex sets}, J. Amer. Math. Soc. \textbf{13} (2000), no.~3, 665--695.
  \MR{1758759}

\bibitem[Whi03]{white2003nature}
\bysame, \emph{The nature of singularities in mean curvature flow of
  mean-convex sets}, J. Amer. Math. Soc. \textbf{16} (2003), no.~1, 123--138.
  \MR{1937202}

\bibitem[Zhu20]{zhu2020so}
Jingze Zhu, \emph{$ SO (2) $ symmetry of the translating solitons of the mean
  curvature flow in $\mathbb{R}^4$}, arXiv:2012.09319 (2020).

\end{thebibliography}

\end{document}